\documentclass[12pt]{amsart}
\usepackage{amsfonts,latexsym,rawfonts,amsmath,amssymb,amsthm,times,a4wide}
\usepackage{graphicx}
\usepackage[plainpages=false]{hyperref}
\usepackage{mathrsfs}
\usepackage{enumerate}
\usepackage{bbm}

\numberwithin{equation}{section}
\DeclareMathOperator{\arccosh}{arccosh}

\usepackage{amscd}
\usepackage{eufrak}
\usepackage{euscript}
\usepackage{epsfig}
\usepackage{array}
\usepackage{enumerate}
\usepackage{color}
\usepackage{wasysym}
\usepackage{pdfsync}
\usepackage{stmaryrd}

\newcommand{\beq}{\begin{equation}}
\newcommand{\eeq}{\end{equation}}
\newcommand{\beqs}{\begin{eqnarray*}}
\newcommand{\eeqs}{\end{eqnarray*}}
\newcommand{\beqn}{\begin{eqnarray}}
\newcommand{\eeqn}{\end{eqnarray}}
\newcommand{\beqa}{\begin{array}}
\newcommand{\eeqa}{\end{array}}

\newtheorem{prop}{Proposition}[section]
\newtheorem{theo}[prop]{Theorem}
\newtheorem{lem}[prop]{Lemma}

\newtheorem{cor}[prop]{Corollary}
\newtheorem{rem}[prop]{Remark}
\newtheorem{ex}[prop]{Example}
\newtheorem{defi}[prop]{Definition}

\newcommand{\R}{{\mathbb R}}
\newcommand{\N}{{\mathbb N}}

\newcommand{\T}{\mathcal{T}}
\newcommand{\wt}{\widetilde}

\newcommand{\tm}{\begin{theo}}
\newcommand{\tmd}{\end{theo}}
\newcommand{\co}{\begin{cor}}
\newcommand{\cod}{\end{cor}}
\newcommand{\prp}{\begin{prop}}
\newcommand{\prpd}{\end{prop}}

\newcommand{\Hmm}[1]{\leavevmode{\marginpar{\tiny%
$\hbox to 0mm{\hspace*{-0.5mm}$\leftarrow$\hss}%
\vcenter{\vrule depth 0.1mm height 0.1mm width \the\marginparwidth}%
\hbox to
0mm{\hss$\rightarrow$\hspace*{-0.5mm}}$\\\relax\raggedright #1}}}

\begin{document}

\title[Combinatorial Ricci flows and the hyperbolization of a class of compact 3-manifolds]{Combinatorial Ricci flows and the hyperbolization of a class of compact 3-manifolds}

\author{Ke Feng}
\address{Ke Feng: School of Mathematical Sciences, University of Electronic Science and Technology of China; No.2006, Xiyuan Ave, West Hi-Tech Zone, Chengdu, Sichuan, 611731, P.R.China}
\email{kefeng@uestc.edu.cn}

\author{Huabin Ge}
\address{Huabin Ge: School of Mathematics, Renmin University of China, Beijing, 100872, P.R. China}
\email{hbge@ruc.edu.cn}

\author{Bobo Hua}
\address{Bobo Hua: School of Mathematical Sciences, LMNS,
Fudan University, Shanghai 200433, China; Shanghai Center for
Mathematical Sciences, Fudan University, Shanghai 200433,
China.}
\email{bobohua@fudan.edu.cn}



\begin{abstract} We prove that for a compact 3-manifold $M$ with boundary admitting an ideal triangulation $\T$ with valence at least 10 at all edges, there exists a unique complete hyperbolic metric with totally geodesic boundary, so that $\mathcal{T}$ is isotopic to a geometric decomposition of $M$.
Our approach is to use a variant of the combinatorial Ricci flow introduced by Luo \cite{[L]} for pseudo 3-manifolds. In this case, we prove that the extended Ricci flow converges to the hyperbolic metric exponentially fast. 
\end{abstract}

\maketitle
\tableofcontents


\section{Introduction}
Suppose that $M$ is a compact irreducible atoroidal Haken 3-manifold whose boundary has zero Euler characteristic, then the interior of $M$ admits a complete hyperbolic metric of finite volume. This is famous Thurston's hyperbolization theorem for Haken manifolds, whose proof is called ``the big monster".  Due to the ``JSJ" decomposition theorem and the Dehn surgery technique, compact 3-manifolds with toric boundary often appear. For a compact 3-manifold with boundary, Thurston's hyperbolization theorem states that it admits a hyperbolic structure if and only if it is irreducible without incompressible tori and atoroidal; see Thurston \cite{[T],[O],[Ka]}. Thurston conjectured that all compact hyperbolic 3-manifolds can be geometrically triangulated. In this paper, under suitable combinatorial assumptions, we confirm this conjecture for such manifolds with higher genus boundary components.
\tm\label{thm:main0}
Let $M$ be a compact 3-manifold with boundary components consisting of surfaces of genus at least 2. If $M$ admits an ideal triangulation $\mathcal{T}$ with valence at least 10 at all edges, then there exists a unique complete hyperbolic metric on $M$ with totally geodesic boundary, so that $\mathcal{T}$ is isotopic to a geometric decomposition of $M$.
\tmd

In the above theorem, the condition is topological and combinatorial, and the conclusion is geometrical. Our approach is based on the combinatorial Ricci flow method which is analytical. It is a large program to hyperbolize 3-manifolds by combinatorial Ricci flows, initiated by Luo \cite{[L]}. The combinatorial Ricci flow aims to find a hyperbolic metric and a corresponding geometric triangulation. To realize the program, we first try to find suitable triangulation (combinatorial) constraints from topological conditions. For instance, one needs to show that a compact 3-manifold admits an ideal triangulation with edge valences at least $10$ under suitable topological conditions, such as ``irreducible without incompressible tori and atoroidal", etc. Then we prove the convergence of certain combinatorial Ricci flow under these combinatorial conditions. That is, one is from topology to combinatorics, and the other is the Ricci flow method with combinatorial restrictions. For closed 3-manifolds and cusped 3-manifolds with torus boundary, the projects are similar.

In geometric analysis, the Ricci flow is a powerful technique to deform the metrics on a manifold, which leads to many important results, e.g. the solution of Poincar\'e's conjecture. 
For a triangulated surface, Chow and Luo \cite{[BL]} introduced the combinatorial Ricci flow to deform the circle packing metrics, and gave an alternative proof of the celebrated Koebe-Andreev-Thurston theorem.
For a compact triangulated 3-manifold with boundary consisting of surfaces of negative Euler characteristic, Luo \cite{[L]} introduced a combinatorial Ricci flow on the set of edges in order to find the complete hyperbolic metric with totally geodesic boundary on the manifold. 
 In this paper, we study a variant Ricci flow analogous to Luo's combinatorial Ricci flow and prove the existence of the polyhedral metric with zero-curvature on edges under some combinatorial condition.

We recall the setting of 3-dimensional triangulated spaces.
Let $\{T_1,\cdots, T_t
\},$ $t\in \N$, be a finite collection of combinatorial tetrahedra and $\mathscr{T}$ be the disjoint union $T_1 \sqcup\cdots \sqcup T_t,$ which is a simplicial complex.
The quotient space $(M,\mathcal{T})=\mathscr{T}/\sim,$ via a family of affine isomorphisms pairing faces of
tetrahedra in $\mathscr{T}$, is called a \emph{compact pseudo 3-manifold} $M$ (together with a triangulation $\mathcal{T}$). Note that simplexes in $\mathcal{T}$ are equivalent classes of simplexes in $\mathscr{T}.$ {Pseudo 3-manifolds are very general concepts, which include manifolds with triangulation as special cases.}
$M$ is called a \emph{closed pseudo 3-manifold} if each codimension-$1$ face of tetrahedra in $\mathscr{T}$ is identified with another codimension-$1$ face. We denote by $V=V(\T)$ (resp. $E=E(\mathcal{T})$) the set of vertices (resp. edges) in $\mathcal{T},$ which are equivalent classes for vertices (resp. edges) of tetrahedra in $\mathscr{T}$ via the gluing. We define the \emph{valence of an edge} $e\in E,$ denoted by $d_e$, to be the number of edges in
$\mathscr{T}$ in the equivalent class of $e.$

A \emph{hyper-ideal tetrahedron} $\sigma$ in $\mathbb{H}^3,$ the hyperbolic $3$-space, is a compact convex polyhedron that is diffeomorphic to a truncated tetrahedron in the 3-dimensional Euclidean space and its four hexagonal faces are right-angled hyperbolic hexagons;
see Figure~\ref{fig:hyperideal1}.
\begin{figure}[htbp]
\begin{center}
\includegraphics[width=0.5\linewidth]{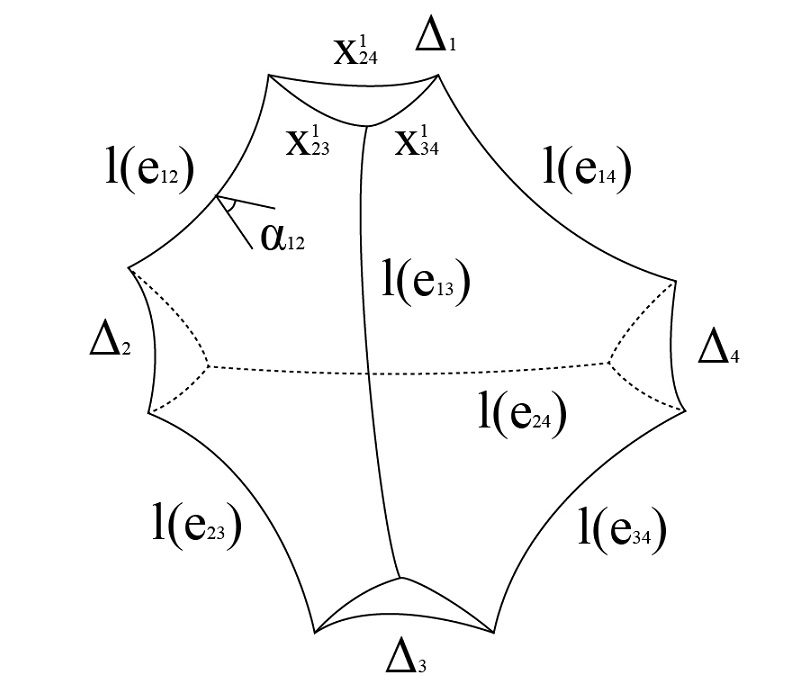}
\caption{\small {Hyper-ideal tetrahedron}}
\label{fig:hyperideal1}
\end{center}
\end{figure}
The four triangular faces, isometric to hyperbolic triangles, are called vertex triangles. An edge in a hyper-ideal tetrahedron is the intersection of two hexagonal faces. The dihedral angle at an edge is the angle between two hexagonal faces adjacent to it.
Let $\Delta_i,$ $i=1,2,3,4,$ be the four vertex triangles of the truncated hyper-ideal tetrahedra $\sigma.$ For $\{i,j\}\subset \{1,2,3,4\},$ we denote by $e_{ij}$ the edge joining $\Delta_i$ to $\Delta_j.$ The length of $e_{ij}$ is denoted by $l_{ij}$ and dihedral angle at $e_{ij}$ is denoted by $\alpha_{ij}$ (always assuming $l_{ij} = l_{ji}$, $\alpha_{ij} = \alpha_{ji}$). The geometry of a hyper-ideal tetrahedron is determined by $(l_{12},l_{13},l_{14},l_{23},l_{24},l_{34})\in \R^6_{>0},$ {the positive orthant of $\R^6.$} The set of isometric classes of hyper-ideal tetrahedra can be described as a subset $\mathcal{L}$ of $\R^6_{>0},$ which is not a convex subset; see Proposition~\ref{REAL}. {From the combinatorial point of view, any hyper-ideal tetrahedron corresponds to a combinatorial tetrahedron $T$ via identifying $\Delta_i$ (resp. $e_{ij}$) with vertices (resp. edges) of $T$. For a finite set $A,$ we denote by $\R^A$ (resp. $\R^A_{>0}$) the set of functions (resp. positive functions) on $A.$ For a fixed ordering of $A,$ each function in $\R^A$ corresponds to a vector in $\R^{|A|},$ where $|A|$ denotes the cardinality of $A.$ 

\begin{defi}\label{def:t1} A \emph{hyper-ideal polyhedral metric}, called \emph{hyper-ideal metric} in short, on $(M,\T)$ is obtained by replacing each tetrahedron in $\T$ by a hyper-ideal tetrahedron and replacing the affine gluing homeomorphisms by isometries preserving the corresponding hexagonal faces. We denote by $l\in \mathbb{R}^E_{>0}$ the edge length vector of the hyper-ideal metric, written as $l=(l(e_1), \dots, l(e_m))$, where $E = \{e_1, \dots, e_m\}.$ The above construction yields a metric space $S(M,\T,l),$ which is uniquely determined by $l.$  We denote by $\mathcal{L}(M, \mathcal{T})\subset \R^E_{>0}$ the set of all hyper-ideal metrics on $(M, \mathcal{T})$ parametrized by the edge length vector $l$. \end{defi}

\begin{defi}\label{defi:realcurv}For a closed pseudo 3-manifold $(M,\T)$ with the edge length vector $l,$ the \emph{Ricci curvature} at each edge $e$ is defined as
\begin{equation}\label{eq:curv}K_e(l)=2\pi-C_e\end{equation}
where $C_e$ is the cone angle at $e,$ i.e. the total dihedral angle in tetrahedra incident to $e.$ {This provides the Ricci curvature vector,
$K=K(l)=(K_{e_1}(l),\cdots,K_{e_m}(l)),$ where $E=\{e_1,\cdots,e_m\}.$ }
\end{defi}

For a closed pseudo 3-manifold $(M,\T),$ let $U(V)$ denote an open regular neighborhood of the set of vertices $V$ in $M.$ {For any given $l\in \mathcal{L}(M, \mathcal{T}),$ by the construction the metric space $S(M,\T,l),$ which is a hyperbolic cone metric with possible singularity on edges, is homeomorphic to $M-U(V)$}. The main purpose is to find cone metrics with no singularity on edges, i.e. $K_e(l)=0$ for all $e\in E,$ {called \emph{zero-curvature hyper-ideal metrics}.}

We recall the motivation for the above construction in the literature; see e.g. \cite{[L]}. Suppose $N$ is a compact 3-manifold with non-empty boundary, whose boundary consists of a union of surfaces of negative Euler characteristic. {The purpose is to find a hyperbolic metric on $N$ with totally geodesic boundary.} Let $C(N)$ be the compact 3-space obtained by coning off each boundary component of $N$ to a point. In particular, if $N$ has $k$ boundary components, then there are exactly $k$ cone points $\{v_1,...,v_k\}$ in $C(N)$ so that $C(N) - \{v_1,...,v_k\}$ is homeomorphic to $N - \partial N.$ An \emph{ideal triangulation} $\T$ of $N$ is a triangulation $\T$
of $C(N)$ such that the vertices of the triangulation are exactly the cone points
$\{v_1,...,v_k\}.$ By Moise \cite{[M]}, every compact 3-manifold $N$ can be ideally triangulated. In our terminology, $(C(N), \T)$ is a closed pseudo 3-manifold
 and $N$ is homeomorphic to $C(N)-st(v_1,...,v_k),$ where $st(v_1,...,v_k)$ is the
open star of the vertices $\{v_1,...,v_k\}$ in the second barycentric subdivision of the
triangulation $\T.$ {As in Definition~\ref{def:t1}, we can endow $(C(N), \T)$ with various hyper-ideal metrics. If there is a hyper-ideal metric $l\in \mathcal{L}(C(N), \mathcal{T})$ with zero Ricci curvature at edges, then we obtain a hyperbolic metric with totally geodesic boundary on $N$ given by the metric space $S(C(N),\T,l)$ constructed in Definition~\ref{def:t1}. This is called a \emph{geometric decomposition} (or \emph{geometric realization}) of a hyperbolic metric on $N$ associated with the ideal triangulation $\T.$}


Motivated by \cite{[BL]}, for a compact 3-manifold with boundary equipped with ideal triangulation, {or more generally a closed pseudo 3-manifolds $(M,\T),$ Luo \cite{[L]} initiated the following combinatorial Ricci flow $l(t)\in \mathcal{L}(M,\T)$ to study the existence of hyperbolic metrics,
\begin{equation}\label{eq:luoflow}\frac{d }{dt}l(t)=K(l(t)),\quad\forall t\geq0.\end{equation}  The vector-valued equation \eqref{eq:luoflow} reads as $$\frac{d }{dt}l(e,t)=K_e(l(t)),\quad \forall e\in E, t\geq0.$$ }This is a negative gradient flow of a locally convex function, related to the co-volume functional, on $\mathcal{L}(M,\T).$ One of main difficulties for the flow approach is that $\mathcal{L}(M,\T)$ is not convex in $\R^E.$

To circumvent the difficulty, Luo and Yang \cite{[LY]} extended the set of hyper-ideal metrics to a general framework.
Given a tetrahedron $\{1,2,3,4\},$ as is shown by \cite{[LY]}, for any $(l_{12},\cdots,l_{34})\in \R^{6}_{>0},$ one can associate it with a {generalized hyper-ideal tetrahedron} such that the extended dihedral angles $\wt{\alpha}_{ij}$, extending dihedral angles for a hyper-ideal tetrahedron, are continuous functions of the edge lengths $l_{ij};$ see Definition~\ref{defi:dihe}. {For any $l\in \R^{6}_{>0}-\mathcal{L},$ it corresponds to a degenerate hyper-ideal tetrahedron; see Section~\ref{sec:pre} for details.

We define generalized hyper-ideal metrics on a compact pseudo 3-manifold $(M,\T)$ following \cite{[LY]}. For any $l\in \R^E_{>0},$ we replace each tetrahedron in $\T$ by a generalized hyper-ideal tetrahedron with edge lengths given by $l,$ and glue them together in the topological sense. Since there are possibly some degenerate hyper-ideal tetrahedra, it may not produce any metric space structure. However, the extended dihedral angles are well defined, which are sufficient for our applications. 

\begin{defi}\label{defi:ghi}  Let $(M,\T)$ be a compact pseudo 3-manifold. We call any $l\in \R^E_{>0}$ a \emph{generalized hyper-ideal metric} on $(M,\T).$ For a closed pseudo 3-manifold, the \emph{generalized Ricci curvature} of an edge $e,$ denoted by $\wt{K}_e(l),$ is defined similarly as in \eqref{eq:curv} by using extended dihedral angles $\wt{\alpha}_{ij},$ and the generalized Ricci curvature vector is denoted by $\wt{K}=\wt{K}(l);$ see Definition~\ref{defi:pre} for the precise definition. \end{defi}
Note that the generalized Ricci curvature $\wt{K}$ extends the Ricci curvature $K$ for hyper-ideal metrics, i.e. $\wt{K}(l)={K}(l)$ for any $l\in \mathcal{L}(M,\T).$

By introducing a change of variables in \eqref{eq:luoflow} and using the ideas in \cite{[GJ],[GJS],[GeH]}, in this paper we study the following {extended} Ricci flow {on the set of generalized hyper-ideal metrics $\R^{E}_{>0}$,}
\begin{equation}\label{eq:newflow}\left\{\begin{array}{ll}\frac{d }{dt}l(t)={\wt K}(l(t)) l(t),& l(t)\in \R^{E}_{>0},  \forall t> 0, \\ l(0)=l_0\in \R^{E}_{>0}.&\end{array}\right.\end{equation}
We say that \emph{the flow \eqref{eq:newflow} converges} if there exists $l_\infty\in \R^E_{>0}$ such that
$$l(t)\to l_\infty,\quad t\to \infty.$$
We prove the long-time existence and uniqueness of the extended Ricci flow.
\tm\label{thm:gp1} For any generalized hyper-ideal metric $l_0\in \R^E_{>0},$ there exists a unique solution of the extended Ricci flow \eqref{eq:newflow} for all time $t\in [0,\infty).$ 
\tmd In the following, we characterize the convergence property of the extended Ricci flow.
\tm\label{thm:expc} For a closed pseudo 3-manifold $(M,\T),$ there exists a zero-curvature hyper-ideal metric if and only if the extended Ricci flow \eqref{eq:newflow} converges to a hyper-ideal metric for some initial data $l_0\in \R^E_{>0}.$ In this case, for any initial data in $\R^E_{>0},$ the extended Ricci flow converges to a hyper-ideal metric exponentially fast .
\tmd

\begin{rem} 
The exponential convergence result suggests that one can compute the zero-curvature hyper-ideal metric using numerical methods effectively.
\end{rem}


The convergence of the extended Ricci flow is the main purpose of the paper. We first give an example.
\begin{ex}\label{ex:12va} There is a 3-manifold $N,$ whose boundary is a surface of genus $2,$ admitting an ideal triangulation $\T;$ see \cite{[MF]}. The triangulation $\T$ consists of $2$ tetrahedra and one edge $e$ with $d_e=12.$ One can show that
 for the initial data $l_0(e)>0,$ the extended Ricci flow \eqref{eq:newflow} converges to a zero-curvature hyper-ideal metric with $l_\infty(e)=\arccosh x_0,$ where $x_0\approx 1.1371$ is the positive solution of {$\frac{x^3+2x^2+x}{2x^3+3x^2-1}=\frac{\sqrt 3}{2}.$} This provides a geometric decomposition of the hyperbolic metric on $N$ associated with $\T.$
\end{ex}


By introducing some combinatorial condition on the valence of edges for a closed pseudo 3-manifold $(M,\T),$ we can prove the convergence of the extended Ricci flow, and hence obtain the existence of zero-curvature hyper-ideal metric. The following are main results of the paper.
For any interval $I\subset (0,\infty),$ we write $$I^E:=\{l\in \R^E_{>0}: l(e)\in I,\ \forall e\in E\}.$$

\tm\label{thm:main10}
Let $(M, \T)$ be a closed pseudo 3-manifold satisfying that $d_e \ge 10$ for all $e\in E.$ Then there exists a zero-curvature hyper-ideal metric $l\in \mathcal{L}(M, \mathcal{T}),$ which is unique in the class {$\R^E_{> 0}.$} Moreover, $$l\in [(3\max_{e\in E} d_e)^{-1},\arccosh 3]^E.$$ 
For any initial data in $\R^E_{> 0},$ the extended Ricci flow \eqref{eq:newflow} converges to $l$ exponentially fast.

\tmd
\begin{rem}
\begin{enumerate}[(i)]
\item There are some closed pseudo 3-manifolds satisfying the combinatorial condition that each edge has valence at least $10,$ e.g. Example~\ref{ex:12va}. Under this condition, we conclude the existence of  the metric with zero Ricci curvature.
\item The metric $l$ we obtained is a hyper-ideal metric, for which it produces a metric space $S(M, \T, l)$ via the gluing.
\item We give a quantitative estimate for the size of the metric with zero Ricci curvature.
\item If there exists a zero-curvature hyper-ideal metric, then it is unique by Luo and Yang's rigidity theorem \cite[Theorem~1.2]{[LY]}; see also Theorem~\ref{thm:uniquezero}.
\end{enumerate}
\end{rem}

We sketch the proof strategy as follows. In step one,
we consider the extended Ricci flow $l(t)$ with a small initial data $l_0\in (0,\arccosh 3)^E.$
We prove that $l(t)$ is uniformly bounded on $[0,\infty),$ i.e. there exists $c>0$ such that $l(t)\in (c,\arccosh 3)^E$ for all $t\geq 0.$ This yields the convergence of the extended Ricci flow, up to a time sequence $t_i\to \infty$ ($i\to \infty$), to the some limit $l_\infty\in \R^E_{>0}.$  For the upper bound estimate, we derive a useful estimate for the dihedral angle at the longest edge of a generalized hyper-ideal tetrahedron; see Corollary~\ref{lem:longest}. {By the dihedral angle estimate, we obtain the upper bound estimate of $l(t)$ in Theorem~\ref{LONG} using the combinatorial condition that $d_e\geq 10$ for all $e\in E.$ The lower bound estimate, Theorem~\ref{bd}, is based on the upper bound estimate and the dihedral angle estimate for the edge with small length; see Proposition~\ref{prop:c0}.} In step two, we need to show that $l_\infty$ is in fact a hyper-ideal metric, i.e. $l_\infty\in \mathcal{L}(M,\T).$ The set $\mathcal{L}(M,\T)$ has been characterized by Luo and Yang \cite{[LY]}; see Proposition \ref{REAL} below. For our purpose, we give a new criterion that $(0,\arccosh 3]^E\subset \mathcal{L}(M,\T);$ see Theorem~\ref{Range}.  Then one can show that $l_\infty$ is a zero-curvature hyper-ideal metric and the other statements follow.


Note that Theorem~\ref{thm:main10} provides a hyperbolic cone metric on a general pseudo 3-manifold, which might not be a manifold. Applying Theorem~\ref{thm:main10} for compact 3-manifolds with non-empty boundary and associated ideal triangulations, we prove Theorem~\ref{thm:main0}. At the end, we give a remark on the result of Theorem~\ref{thm:main0}.

\begin{rem}Using tricky topological arguments, Costantino, Frigerio, Martelli and Petronio \cite{[CFMP]} proved the following related result for Theorem~\ref{thm:main0}: if a 3-manifold $M$ has an ideal triangulation $\T$ whose edges have valence at least 6, then $M$ admits a hyperbolic metric with totally geodesic boundary and cusped ends, and the edges are homotopically non-trivial with respect to the boundary of $M,$ whence homotopic to geodesics. Although $M$ admits a hyperbolic metric $g,$ it is not known whether there is any geometric decomposition (or realization) of the hyperbolic metric $g$ associated with the given triangulation $\T.$ It is possible that $g$ is realized by another triangulation $\T',$ not by $\T.$ They conjectured that that if $M$ has an ideal triangulation $\T$ whose edges have valence at least 6, then $\T$ is realized by hyperbolic partially truncated tetrahedra; see \cite[Conjecture~1.8]{[CFMP]}. This is easily verified for the cases of triangulations whose edges have same valence, but not known in general. In this paper, we use the combinatorial Ricci flow method to partially confirm the conjecture for triangulations with edge valences at least $10;$ see Theorem~\ref{thm:main0}. Moreover, in this case one can find a geometric decomposition of the hyperbolic metric $g$ associated with  $\T$ using the extended Ricci flow~\eqref{eq:newflow} by Theorem~\ref{thm:main10}. See \cite{[Ko],[La2],[La1],[LDD],[GHR]} for more constraints of the triangulation and topology on 3-manifolds.
\end{rem}

The paper is organized as follows.
In the next section, we recall some results of generalized hyper-ideal tetrahedra obtained by \cite{[LY]}. In Section~\ref{sec:gat}, we prove some new geometric properties for a generalized hyper-ideal tetrahedron. In Section~\ref{sec:ecf}, we study general properties for the extended Ricci flow \eqref{eq:newflow} and prove Theorem~\ref{thm:gp1}. In the last section, we prove the main results, Theorem~\ref{thm:main0}, Theorem~\ref{thm:expc}, Theorem~\ref{thm:main10}. 

\bigskip

\bigskip

\section{Preliminaries}\label{sec:pre}

In this section, we recall some results on the geometry of a hyper-ideal (or generalized hyper-ideal) tetrahedron obtained by {\cite{[BB],[Sch],[FP],[Riv1],[CGV],[LY]}.}
\subsection{{Generalized hyper-ideal tetrahedra}}

A hyper-ideal tetrahedron $\sigma$ in $\mathbb{H}^3$ is a compact polyhedron that is diffeomorphic to a truncated tetrahedron in $\mathbb{R}^3$. An edge in a hyper-ideal tetrahedron is the intersection of two hexagonal faces, and a vertex edge is the intersection of one hexagonal face and one triangular face.  A hyper-ideal tetrahedron has the following properties: firstly, its four hexagonal faces are right-angled hyperbolic hexagons; secondly, its four triangular faces, called vertex triangles, are isometric to hyperbolic triangles; thirdly, the dihedral angle between a hexagonal face and a vertex triangle is $\frac{\pi}{2}$, and the angle between two hexagonal faces adjacent to one edge is called the dihedral angle at the edge.

{The following is another description of hyper-ideal tetrahedra by \cite{[CGV]}.} Let $\mathbb{K}^3 \subset \mathbb{R}^3$ be the open ball representing $\mathbb{H}^3$ via the  Klein model. Then we can obtain a hyper-ideal tetrahedron by the following process. Let $\mathscr{P} \subset \mathbb{R}^3$ be a convex Euclidean tetrahedron such that each vertex $v_i (i = 1,2,3,4)$ lies in $\mathbb{R}^3\backslash \mathbb{K}^3$ and each edge intersects $\partial \mathbb{K}^3$. Let $C_i$ be the cone with the apex $v_i$ tangent to $\partial \mathbb{K}^3$ and $\pi_i$ be the half-space not containing $v_i$ such that $\partial \pi_i \cap \partial \mathbb{K}^3 = C_i \cap \partial \mathbb{K}^3$. Then, a hyper-ideal hyperbolic tetrahedron is given by $P := \mathscr{P}\cap \bigcap_i\pi_i;$ see Figure ~\ref{fig:k}.
\begin{figure}[htbp]
\begin{center}
\includegraphics[width=0.6\linewidth]{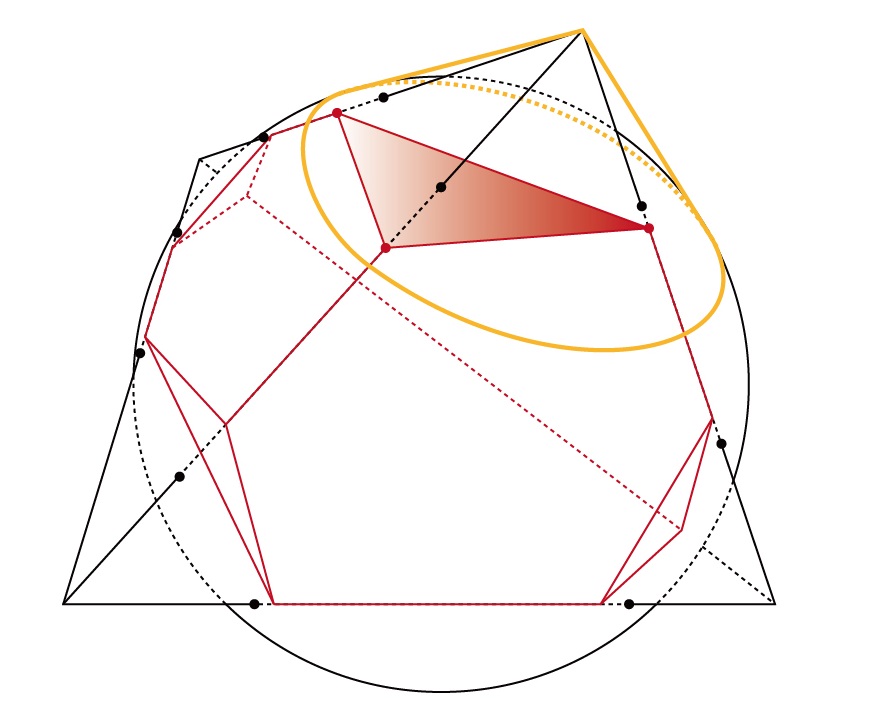}
\caption{\small {A hyper-ideal tetrahedron in the Klein model}}
\label{fig:k}
\end{center}
\end{figure}

Next, we recall some important results about the parameterization of the set of hyper-ideal tetrahedra and its extension for the set of generalized hyper-ideal tetrahedra; see \cite{[LY]} for details. Let $\{1,2,3,4\}$ be a combinatorial tetrahedron. The indices $i,j,k,h\in\{1,2,3,4\}$ are always considered to be distinct in this paper.  For a hyper-ideal tetrahedron $\sigma$ based on $\{1,2,3,4\},$ we denote by $\Delta_i$ ($e_{ij},$ $l_{ij}$ and $\alpha_{ij},$ resp.) a vertex triangle (an edge, the edge length, and a dihedral angle, resp.) as in the introduction, and by $H_{ijk}$ the hexagonal face adjacent to $e_{ij}, e_{ik}$ and $e_{jk}.$  The length of the vertex edge $\Delta_i \cap H_{ijk}$ is denoted by $x_{jk}^i$.
\begin{prop}[\cite{[BB],[F]}]
Let $\sigma$ be a hyper-ideal tetrahedron.
\begin{itemize}
\item The isometry class of $\sigma$ is determined by its dihedral angle vector $(\alpha_{12},\cdots,\alpha_{34})$ in $\mathbb{R}^6$, which satisfies the condition that $\alpha_{ij} > 0$ for any $\{i,j\}\subset\{1,2,3,4\}$, and $\sum_{j \neq i} \alpha_{ij} < \pi$ for any fixed $1\leq i\leq 4.$
\item Conversely, given any $(\alpha_{12}, \cdots, \alpha_{34}) \in \mathbb{R}^6_{>0}$ satisfying that $\sum_{j \neq i} \alpha_{ij} < \pi$ for each $1\leq i\leq 4,$ there exists a hyper-ideal tetrahedron whose dihedral angles are given by $\alpha_{ij}.$
\item The isometry class of $\sigma$ is also determined by its edge length vector $(l_{12}, \cdots, l_{34}) \in \mathbb{R}^6_{> 0}$.
\end{itemize}
\end{prop}

Thus, the set of isometry classes of hyper-ideal tetrahedra parameterized by dihedral angles is the open convex polytope in $\mathbb{R}^6,$
\beq
\mathscr{B} = \big\{ (\alpha_{12}, \cdots, \alpha_{34}) \in \mathbb{R}^6_{>0}\big|\sum_{j \neq i} \alpha_{ij} < \pi, \  \text{for all}\ 1\leq i\leq 4\big\}.
\eeq

Let $vol(\cdot)$ denote the hyperbolic volume of a hyper-ideal tetrahedron. {Due to Casson-Rivin's angle structure theory (see for example \cite{[Riv3],[Riv2],[Luo3],[Riv1],[GF],[HR],[Luo1]}), it can be naturally regarded as a function} $vol: \mathscr{B} \to \mathbb{R}$, and satisfies the Schl\"afli formula {(see \cite{[Bon]} for more general setting)},
\[\frac{\partial vol}{\partial  \alpha_{ij}} = -\frac{l_{ij}}{2}.\] Some other properties of this volume function can be found in \cite{[Riv1],[Sch]}.

Let $\mathcal{L}$ denote the set of vectors $(l_{12}, \cdots, l_{34})\subset\R^6_{>0}$ such that there exists a hyper-ideal tetrahedron having $l_{ij}$ as the length of the edge $e_{ij}$ for any $\{i,j\}\subset\{1,2,3,4\}.$ The volume $vol(\cdot)$ can be also regarded as a function on $\mathcal{L}.$ The Legendre transform of $vol$, called the co-volume functional, is given by
\[cov(l) = 2vol(l) + \sum_{i < j}\alpha_{ij}l_{ij}, \quad l\in\mathcal{L}.\]

\begin{prop}[Corollary~5 in \cite{[L]}]\label{prop:strictconvex}The functional $cov:\mathcal{L}\to \R$ is a smooth function, which has a positive definite Hessian matrix at each $l\in \mathcal{L},$ and hence it is locally strictly convex.
\end{prop}
Since $\mathcal{L}$ is not a convex subset of $\R^6,$ $cov$ may not be globally convex. It is useful to extend the co-volume functional to a $C^1$-smooth and convex function on $\mathbb{R}^6_{\ge 0}$ (or $\mathbb{R}^6);$ see \cite{[LY]}.

Firstly, the following are the formulae of dihedral angles in terms of the edge length vector in $\mathcal{L};$ see \cite[Proposition~3.1]{[Luo3]}} and \cite[Lemma~4.3]{[LY]}.
\begin{lem}\label{l1}
For $l=(l_{12}, \cdots, l_{34}) \in \mathcal{L}$ and $\{i, j, k, h\} = \{1, 2, 3, 4\}$, let $l_{ij} = l_{ji}$ for $i \neq j.$ Set
\beq \label{x}
x_{jk}^i = \arccosh\big ( \frac{\cosh l_{ij}\cosh l_{ik} + \cosh l_{jk}}{\sinh l_{ij} \sinh l_{ik}}\big )
\eeq
and
\beq \label{phi}
\phi_{kh}^i = \frac{\cosh x_{jk}^i \cosh x_{jh}^i - \cosh x_{kh}^i}{\sinh x_{jk}^i \sinh x_{jh}^i}.
\eeq
Then $\phi_{kh}^i(l) = \phi_{kh}^j(l).$ Define the function $\phi_{ij}: \mathcal{L} \to \mathbb{R}$ by $\phi_{ij}(l) = \phi_{kh}^i(l)$. Then $\cos\alpha_{ij}=\phi_{ij}$ and
\beq \label{angle}
\phi_{ij} = \frac{c_{ik}c_{ih} + c_{jk}c_{jh} + c_{ij}c_{ik}c_{jh} + c_{ij}c_{ih}c_{jk} -s_{ij}^2c_{kh}}{\sqrt{2c_{ij}c_{ik}c_{jk} + c_{ij}^2 + c_{ik}^2 + c_{jk}^2 - 1}\sqrt{2c_{ij}c_{ih}c_{jh} + c_{ij}^2 + c_{ih}^2 + c_{jh}^2 - 1}},
\eeq
where $c_{ij} = \cosh l_{ij}$, $s_{ij} = \sinh l_{ij}$.
\end{lem}

Note that for any $l = (l_{12}, \cdots, l_{34}) \in \mathcal{L}$ as the edge length vector of a hyper-ideal tetrahedron, $x_{jk}^i$ and $\arccos(\phi_{ij})=\alpha_{ij}$ are the length of the vertex edge $\Delta_i \cap H_{ijk}$ and the dihedral angle at $e_{ij}$ respectively. 

Since $\mathcal{L}$ is not convex in $\R^6_{>0},$ Luo and Yang \cite{[LY]} introduced the generalized hyper-ideal tetrahedra to extend the subset $\mathcal{L}$. A \emph{generalized hyper-ideal tetrahedron} is a topological truncated tetrahedron so that each edge $e_{ij}$ is assigned a positive number $l_{ij},$ called the edge length. The set of generalized hyper-ideal tetrahedra is parametrized by the edge length vector $l\in\R^6_{>0}.$
If $l\in \mathcal{L},$ then it corresponds to a (real) hyper-ideal tetrahedron. Otherwise, for $l\in \R^6_{>0}\setminus \mathcal{L},$ it corresponds to a degenerate
hyper-ideal tetrahedron. Moreover, if each edge $e_{ij}$ is assigned a nonnegative number $l_{ij}$, then we can use $\R^6_{\geq 0}$ to parametrize a larger class of generalized hyper-ideal tetrahedra, called \emph{generalized hyper-ideal tetrahedra in the wide sense}.

\begin{defi}\label{defi:dihe}Let $l\in \R^6_{\geq 0}$ be a generalized hyper-ideal tetrahedron in the wide sense. For any $\{i,j\}\subset \{1,2,3,4\},$ we use the equation \eqref{angle} to define $\phi_{ij},$ and 
define $$\alpha_{ij} = \arccos (-1 \vee \phi_{ij} \wedge1).$$ We call $\alpha_{ij}$ the \emph{extended dihedral angle} at the edge $e_{ij}.$
\end{defi}
The extended dihedral angle $\alpha_{ij}$ was written as $\wt{\alpha}_{ij}$ in the introduction to be distinguished from the usual dihedral angle. In the rest of the paper, for simplicity we call it the dihedral angle if it does not cause any confusion in the context. Note that $\phi_{ij}$ and $\alpha_{ij}$ are continuous functions on $\mathbb{R}_{\ge 0}^6.$ Moreover, for a hyper-ideal tetrahedron, $\alpha_{ij}$ equals to the usual dihedral angle at the edge $e_{ij}.$

We introduce a special class of generalized hyper-ideal tetrahedra. A \emph{flat hyper-ideal tetrahedron} is defined as follows; see {Figure~\ref{fig:flat}.} Take a right-angled hyperbolic octagon $Q$ with eight edges cyclically labelled as $\Delta_1,e_{12}, \Delta_2,e_{23},\Delta_3,e_{34},\Delta_4,e_{41}.$ Let $e_{13}$ (and $e_{24}$) be the shortest geodesic arc in $Q$ joining $\Delta_1$ to $\Delta_3$ (and $\Delta_2$ and $\Delta_4$). We call $(Q,\{e_{ij}\})$ a flat hyper-ideal tetrahedron with six edges $e_{ij}$.

\begin{figure}[htbp]
\begin{center}
\includegraphics[width=0.5\linewidth]{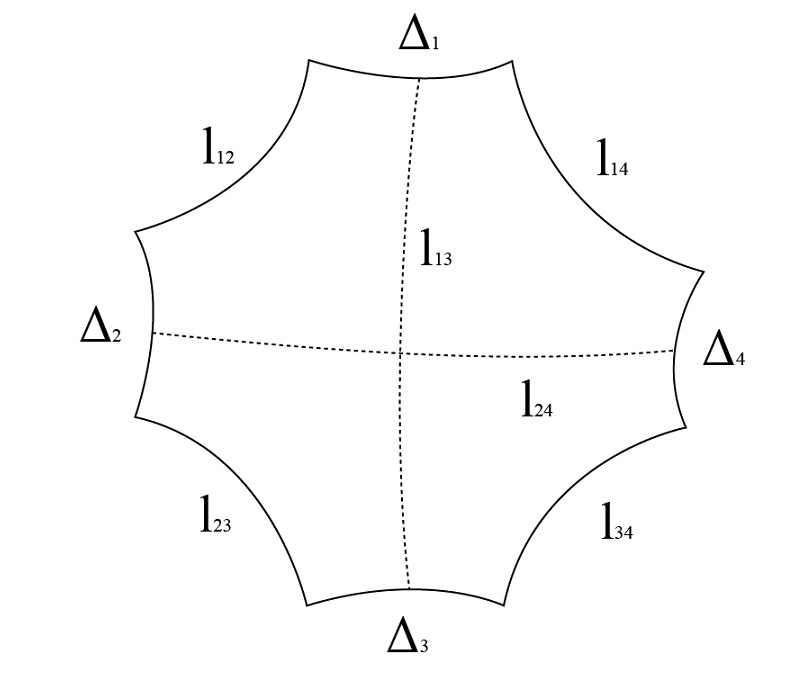}
\caption{\small {Flat hyper-ideal tetrahedron}}
\label{fig:flat}
\end{center}
\end{figure}
In particular, the edge lengths $l_{ij}$ and the dihedral angles $\alpha_{ij}$ are well-defined. The dihedral angles at $e_{13}$ and $e_{24}$ are $\pi$ and are $0$ for all other edges. Note that some configurations of degenerate hyper-ideal tetrahedra in $\R^6_{>0}-\mathcal{L}$ are in fact flat hyper-ideal tetrahedra. 


In $\R^6_{\geq 0},$ the set $\mathcal{L}$ can be characterized by the range of $\phi_{ij}.$
\begin{prop}[Proposition~4.4 and Lemma~4.7 in \cite{[LY]}]\label{REAL}
The set of all hyper-ideal tetrahedra parametrized by the edge lengths is
\[\mathcal{L} = \{l \in \mathbb{R}_{\geq 0}^6| \phi_{ij}(l) \in (-1, 1) \ \text{for all} \ \{i,j\} \subset \{1, 2, 3, 4\}\}.\]
\end{prop}


\subsection{Co-volume of a generalized hyper-ideal tetrahedron}
{The volume is naturally defined as a function of dihedral angles. It is useful to consider a co-volume function, which is a function of edge lengths and can be realized by the Legendre-Fenchel dual of the volume function. The co-volume function is quite amazing (especially in the ideal tetrahedron setting), which has the form dates back to \cite{[Col],[CKP],[BPS]}. A concrete calculation of volume can be seen in \cite{[U]}. Recall that} Luo and Yang \cite{[LY]} extended  the co-volume functional $cov:\mathcal{L}\to \R$ to the set of generalized hyper-ideal tetrahedra.

For the co-volume functional $cov:\mathcal{L}\to \R,$ by the Schl\"afli formula,
\[\frac{\partial cov}{\partial l_{ij}} = \alpha_{ij}\]
for $i \neq j$, where $\alpha_{ij}: \mathcal{L} \to \mathbb{R}$ is the dihedral angle function at the edge $e_{ij}.$ In particular, the differential 1-form $\omega = \sum_{i < j} \alpha_{ij} dl_{ij} = d cov$ is a closed form in $\mathcal{L}$, and co-volume can be recovered via the integration $cov(l) = \int^l \omega;$ see \cite{[L]}.

For each $l =(l_{12}, \dots, l_{34}) \in \mathbb{R}^6$, let $l^+ := (l_{12}^+, \dots, l_{34}^+) \in \mathbb{R}_{\ge 0}^6$ where $l_{ij}^+ = l_{ij}\vee 0$. Moreover, we can extend the function $\alpha_{ij}: \mathcal{L} \to \mathbb{R}$ to a continuous function $\alpha_{ij}: \mathbb{R}^6 \to \mathbb{R}$ by $$\alpha_{ij}(l) = \alpha_{ij}(l^+),$$ and call $\alpha(l) = (\alpha_{12}(l), \dots, \alpha_{34}(l))$ the dihedral angle vector of $l \in \mathbb{R}^6.$
This defines a new continuous 1-form $\mu$ on $\mathbb{R}^6$ by
\begin{equation}\label{eq:mu1}\mu(l) = \sum_{i \neq j} \alpha_{ij}(l)dl_{ij}.\end{equation}

\begin{prop}[Proposition~4.10 in \cite{[LY]}]\label{form}
The continuous differential 1-form $\mu(l) = \sum_{i \neq j} \alpha_{ij}(l)dl_{ij}$ is closed in $\mathbb{R}^6$, that is, for any Euclidean triangle $\Delta$ in $\mathbb{R}^6$, $\int_{\partial \Delta} \mu = 0$.
\end{prop}

We define the functional $cov: \mathbb{R}^6 \to \mathbb{R}$ by the integral
\beq \label{cov}
cov(l) = \int_{(0, \dots, 0)}^l \mu + cov(0, \dots, 0),
\eeq where $cov(0, \dots, 0)=16\Lambda(\frac{\pi}{4})$, by the result of Ushijima \cite{[U]}. Here $\Lambda(a)=\int_0^a \ln|2\sin t|dt$ is the Lobachevsky function. Note that the functional $cov$ defined above extends the co-volume functional $cov:\mathcal{L}\to \R.$
\begin{prop}[Corollary~4.12 in \cite{[LY]}]\label{covx}
The functional $cov: \mathbb{R}^6 \to \mathbb{R}$ is a $C^1$-smooth convex function.
\end{prop}

\section{Geometric properties for a generalized hyper-ideal tetrahedron}\label{sec:gat}


In this section, we derive some new estimates for geometric quantities of a generalized hyper-ideal tetrahedron, which will be crucial for our applications. Throughout the section, we only consider a single generalized hyper-ideal tetrahedron $\sigma$ with the edge length vector $(l_{12}, \cdots, l_{34})\in \R^6_{
\geq 0}.$

To simplify the notation, we write $$(e_1,e_2,e_3,e_4,e_5,e_6):=(e_{12},e_{13},e_{14},e_{34},e_{24},e_{23}),$$ where $e_{ij}$ are the edges of $\sigma.$ By the above correspondence, e.g. $\{1\}$ with $\{12\},$ we write the quantities on $e_{ij}$ such as $l_{ij}, \alpha_{ij}$ and $\phi_{ij}$ using the index $\{1,\cdots,6\},$ e.g. $l_1=l_{12},$ etc.; see Figure ~\ref{fig:hyperideal}. In the following, we set
$$\alpha:=\alpha_1=\alpha_{12},\quad \phi:=\phi_1=\phi_{12}.$$



\begin{figure}[htbp]
\begin{center}
\includegraphics[width=0.5\linewidth]{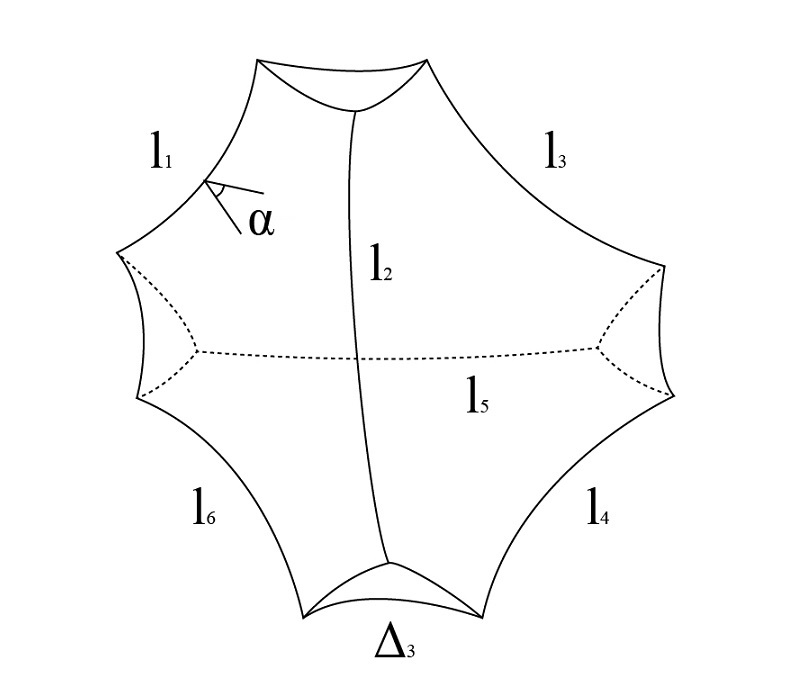}
\caption{\small {A hyper-ideal tetrahedron}}
\label{fig:hyperideal}
\end{center}
\end{figure}

By \eqref{angle},
\beq \label{PHI}
\phi =
\frac{x_2x_3 + x_5x_6 + x_1x_2x_5 + x_1x_3x_6 - x ^2_1x_4 + x_4}{\sqrt {2x_1x_2x_6 + x^2_1 + x^2_2 + x^2_6 - 1}\sqrt {2x_1x_3x_5 + x^2_1 + x^2_3 + x^2_5 - 1}},
\eeq
where $x_i = \cosh l_i$.

Note that for any $l\in \R^6_{\geq 0},$ $x\in \R^6_{\geq 1}:=[1,\infty)^6.$  Note that $\phi$ is monotonely decreasing in $x_4$ since $x_1\geq1.$
By taking partial derivatives of $\phi,$ we get

\begin{equation}\label{eq:pde}\left\{\begin{array}{c}
\frac{\partial \phi}{\partial x_2} = A_0[(x_1x_4 + x_2x_5 - x_3x_6)x_6 + x_3 + x_1x_5 +x_2x_4]\\
\frac{\partial \phi}{\partial x_5} = {A_1}[(x_1x_4 + x_2x_5 - x_3x_6)x_3 + x_6 + x_1x_2 +x_5x_4]\\
\frac{\partial \phi}{\partial x_3} = A_1[(x_1x_4 - x_2x_5 + x_3x_6)x_5 + x_2 + x_1x_6 +x_3x_4],\\
\frac{\partial \phi}{\partial x_6} = {A_0}[(x_1x_4 - x_2x_5 + x_3x_6)x_2 + x_5 + x_1x_3 +x_6x_4],
\end{array}\right.\end{equation}
where 
\[A_0 = ({x_1}^2 - 1)(2x_1x_2x_6 + x^2_1 + x^2_2 + x^2_6 - 1)^{-3/2}(2x_1x_3x_5 + x^2_1 + x^2_3 + x^2_5 - 1)^{-1/2} \geq 0,\]
\[A_1 = ({x_1}^2 - 1)(2x_1x_2x_6 + x^2_1 + x^2_2 + x^2_6 - 1)^{-1/2}(2x_1x_3x_5 + x^2_1 + x^2_3 + x^2_5 - 1)^{-3/2} \geq 0.\]

We derive some properties for the function $\phi.$ The function $\phi$ has the following symmetry. The proof is evident, and hence we omit it.
\begin{prop}\label{prop:sym} For any $x\in \R^6_{\geq 1},$
$$\phi(x)=\phi(x_1,x_3,x_2,x_4,x_6,x_5)=\phi(x_1,x_5,x_6,x_4,x_2,x_3).$$
\end{prop}

Moreover,  at least one of the pairs of $(\frac{\partial \phi}{\partial x_2},\frac{\partial \phi}{\partial x_5})$
and $(\frac{\partial \phi}{\partial x_3},\frac{\partial \phi}{\partial x_6})$ has the positivity property as follows.

\begin{prop}\label{prop:twop}  For any $x\in \R^6_{\geq 1}$ with $x_1>1,$ one of the following holds:
\begin{enumerate}[(i)]
\item $\frac{\partial \phi}{\partial x_2}>0$ and $\frac{\partial \phi}{\partial x_5}>0$;
\item $\frac{\partial \phi}{\partial x_3}>0$ and $\frac{\partial \phi}{\partial x_6}>0$.
\end{enumerate}
\end{prop}
\begin{proof} By \eqref{eq:pde}, if $x_2x_5\geq x_3x_6,$ then the assertion $(i)$ holds. Otherwise, the assertion $(ii)$ holds.
\end{proof}

Next, we derive the monotonicity for the partial derivatives $\frac{\partial \phi}{\partial x_i},$ $i=\{2,3,5,6\}.$ For any such $i,$ we define the index $\hat{i},$ which is in pair of $i,$ as follows,
\begin{equation*}\hat{i}=\left\{\begin{array}{cc}
5, & i=2,\\
6, & i=3,\\
2, & i=5,\\
3, & i=6.
\end{array}\right.\end{equation*} For any $k\in \{1,\cdots, 6\},$ we denote by $m_k=(0,\cdots, 1,\cdots, 0)$ the $k$-th coordinate unit vector in $\R^6.$
\begin{prop}\label{prop:mon1} Let $x\in \R^6_{\geq 1}$ and $i\in \{2,3,5,6\}.$
\begin{enumerate}[(i)]
\item If $\frac{\partial \phi}{\partial x_i}(x)\geq 0,$ then for any $t,s\geq 0,$ $$\frac{\partial \phi}{\partial x_i}(x+tm_i+sm_{\hat{i}})\geq 0.$$ In particular,
$\phi(x)\leq \phi(y),$ for any $y$ satisfying $y_i\geq x_i, y_j=x_j$ ($\forall j\neq i$).
\item If $\frac{\partial \phi}{\partial x_i}(x)\leq 0,$ then for any $t,s\leq 0$ satisfying $x+tm_i+sm_{\hat{i}}\in \R^6_{\geq 1},$ $$\frac{\partial \phi}{\partial x_i}(x+tm_i+sm_{\hat{i}})\leq 0.$$
 In particular,
{$\phi(x)\leq \phi(y),$} for any $y$ satisfying $1\leq y_i\leq x_i, y_j=x_j$  ($\forall j\neq i$).
\end{enumerate}
\end{prop}
\begin{rem} By the above result, for any $i\in \{2,3,5,6\},$ and fixed $x_j$ ($\forall j\neq i$), as a one-variable function of $x_i,$ $\phi(x)$ is monotone in some interval of $x_i.$
\end{rem}
\begin{proof} Without loss of generality, we prove the result for $i=2.$ We may assume that $x_1>1,$ otherwise the result is trivial. By \eqref{eq:pde}, since $A_0\geq 0,$ the sign of $\frac{\partial \phi}{\partial x_i}$ is same as that of
the term in the bracket $[\cdot],$
$$G(x):=(x_1x_4 + x_2x_5 - x_3x_6)x_6 + x_3 + x_1x_5 +x_2x_4.$$ Note that $G(x)$ is monotonely increasing in $x_2$ and $x_5$ for fixed $x_1,x_3,x_4,x_6.$

For the assertion $(i),$ since $\frac{\partial \phi}{\partial x_2}(x)\geq 0$ and $x_1>1,$ $G(x)\geq 0.$
By the monotonicity of $G,$ for any $t,s\geq 0,$ $$G(x+tm_2+sm_5)\geq G(x)\geq 0.$$
This yields that $\frac{\partial \phi}{\partial x_2}(x+tm_2+sm_5)\geq 0.$

Moreover, for the case $s=0,$ $$\frac{\partial \phi}{\partial x_2}(x+tm_2)\geq 0,\ \forall t\geq 0.$$ This implies that for fixed $x_j$ ($\forall j\neq 2$), as a one-variable function $\phi(x_1,\cdot,x_3,x_4,x_5,x_6)$ is monotonely increasing on $(x_2,\infty).$ This proves $(i).$

By the same arguments, one can prove $(ii).$
\end{proof}

This yields the following corollary.
\co\label{coro:mon2}
Let $x\in \R^6_{\geq 1}$ and $i\in \{2,3,5,6\}.$
 If $\frac{\partial \phi}{\partial x_i}(x)\geq 0$ and $\frac{\partial \phi}{\partial x_{\hat{i}}}(x)\geq 0,$ then for any $t,s\geq 0,$ $$\phi(x)\leq\phi(x+tm_i+sm_{\hat{i}}).$$

\cod
\begin{proof} Without loss of generality, we prove the case $i=2.$ 
 By Proposition~\ref{prop:mon1}, for any $t,s\geq 0,$ $$\frac{\partial \phi}{\partial x_2}(x+tm_2+sm_5)\geq 0, \frac{\partial \phi}{\partial x_5}(x+tm_2+sm_5)\geq 0.$$ For fixed $x_1,x_3,x_4,x_6,$ as a two-variable function $\phi(x_1,\cdot,x_3,x_4,\cdot,x_6)$ is monotonely increasing in $x_2$ and $x_5$ respectively. This yields the result.
\end{proof}

Now we are ready to prove some useful estimates.
\tm\label{thm:kest1} Let $x\in \R^6_{\geq 1}$ satisfying $x_i\leq a,\ \forall ~ 2\leq i\leq 6$ for some $a>1.$ Then
$$\phi(x) \le \max\{\phi(x_1,a, a, 1, a, a), \phi(x_1,a, 1, 1, a, 1),\phi(x_1,a, a, 1, a, 1) \}.$$
\tmd
\begin{proof} Noting that $\phi$ is monotonely decreasing in $x_4,$ we have
$$\phi(x)\leq \phi(x_1,x_2,x_3,1,x_5,x_6).$$ By Proposition~\ref{prop:twop}, at the point $(x_1,x_2,x_3,1,x_5,x_6),$ either $\frac{\partial \phi}{\partial x_2}>0$ and $\frac{\partial \phi}{\partial x_5}>0,$ or
\item $\frac{\partial \phi}{\partial x_3}>0$ and $\frac{\partial \phi}{\partial x_6}>0.$ By the symmetry of $\phi,$ Proposition~\ref{prop:sym}, without loss of generality, we may assume that $\frac{\partial \phi}{\partial x_2}>0$ and $\frac{\partial \phi}{\partial x_5}>0.$ By Corollary~\ref{coro:mon2},
$$\phi(x_1,x_2,x_3,1,x_5,x_6)\leq \phi(x_1,a,x_3,1,a,x_6).$$

Considering $\frac{\partial \phi}{\partial x_3}$ at the point $(x_1,a,x_3,1,a,x_6),$ we have the following cases:
\begin{description}
  \item[Case 1] $\frac{\partial \phi}{\partial x_3}(x_1,a,x_3,1,a,x_6)\geq 0.$ By $(i)$ in Proposition~\ref{prop:mon1},
  $$\phi(x_1,a,x_3,1,a,x_6)\leq \phi(x_1,a,a,1,a,x_6).$$ Now we consider $\frac{\partial \phi}{\partial x_6}$ at the point $(x_1,a,a,1,a,x_6)$ and divide it into subcases.
\begin{description}
  \item[Case 1.1] $\frac{\partial \phi}{\partial x_6}(x_1,a,a,1,a,x_6)\geq 0.$ By $(i)$ in Proposition~\ref{prop:mon1},
  $$\phi(x_1,a,a,1,a,x_6)\leq \phi(x_1,a,a,1,a,a).$$

   \item[Case 1.2] $\frac{\partial \phi}{\partial x_6}(x_1,a,a,1,a,x_6)< 0.$ By $(ii)$ in Proposition~\ref{prop:mon1},
  $$\phi(x_1,a,a,1,a,x_6)\leq \phi(x_1,a,a,1,a,1).$$
  \end{description}

  \item[Case 2] $\frac{\partial \phi}{\partial x_3}(x_1,a,x_3,1,a,x_6)< 0.$
  By $(ii)$ in Proposition~\ref{prop:mon1},
  $$\phi(x_1,a,x_3,1,a,x_6)\leq \phi(x_1,a,1,1,a,x_6).$$ Now we consider $\frac{\partial \phi}{\partial x_6}$ at the point $(x_1,a,1,1,a,x_6)$ and divide it into subcases.
\begin{description}
  \item[Case 2.1] $\frac{\partial \phi}{\partial x_6}(x_1,a,1,1,a,x_6)\geq 0.$ By $(i)$ in Proposition~\ref{prop:mon1},
  $$\phi(x_1,a,1,1,a,x_6)\leq \phi(x_1,a,1,1,a,a).$$

   \item[Case 2.2] $\frac{\partial \phi}{\partial x_6}(x_1,a,1,1,a,x_6)< 0.$ By $(ii)$ in Proposition~\ref{prop:mon1},
  $$\phi(x_1,a,1,1,a,x_6)\leq \phi(x_1,a,1,1,a,1).$$
  \end{description}

  \end{description}

Combining all cases above, noting that $\phi(x_1,a,a,1,a,1)=\phi(x_1,a,1,1,a,a),$ we prove the theorem.

\end{proof}

This result yields the estimate for the dihedral angle at the longest edge of a generalized hyper-ideal tetrahedron.
\co\label{lem:longest}

Let $\sigma$ be a generalized hyper-ideal tetrahedron in the wide sense with the edge length vector $l\in \R^6_{\geq 0}.$ Suppose that $l_1=\max_{1\leq i\leq 6}l_i,$ then
$$\cos\alpha_1\leq \max\{\phi(x_1,x_1, x_1, 1, x_1, x_1), \phi(x_1, x_1, 1, 1, x_1, 1),\phi(x_1, x_1, x_1, 1, x_1, 1) \},$$ where $x_1=\cosh l_1.$ Moreover, if $l_1=\arccosh 3,$ then
$$\alpha_1> \frac{\pi}{5}.$$
\cod
\begin{rem} The condition $l_1=\arccosh 3$ is properly chosen for our further applications.
\end{rem}
\begin{proof} For $x_i=\cosh l_i,$ $1\leq i\leq 6,$ $x_1=\max_{1\leq i\leq 6}x_i.$ Setting $a=x_1$ in Theorem~\ref{thm:kest1}, we prove the first assertion.

For $l_1=\arccosh 3,$ $x_1=3.$ This yields that
$$\cos\alpha_1\leq \frac{4}{5}.$$ This implies that $\alpha_1> \frac{\pi}{5}.$
\end{proof}

Moreover, we prove the following theorem.
\begin{theo}\label{Range}
Let $\sigma$ be a generalized hyper-ideal tetrahedron with the edge length vector $l\in \R^6_{>0}.$ If $l_i\leq \arccosh 3,$ for all $1\leq i\leq 6,$ then  $$\phi_i(l)\in (-1,1)\quad \mathrm{for\ all}\ 1\leq i\leq 6.$$ In particular, $\sigma$ is a hyper-ideal tetrahedron, i.e. $l\in \mathcal{L}.$
\end{theo}
\begin{proof} For the first assertion, we prove the upper and lower bound estimate for $\phi_i$ respectively. Let $x_i=\cosh(l_i),$ $1\leq i\leq 6.$ This yields that $x_i\in (1,3]$ for all $i.$ Without loss of generality, it suffices to consider $i=1$ by the symmetry. Let $\phi=\phi_1.$

For the upper bound estimate of $\phi,$ by setting $a=3$ in Theorem~\ref{thm:kest1},
\begin{eqnarray*}\phi(x)&\leq& \max\{\phi(x_1,3, 3, 1, 3, 3), \phi(x_1, 3, 1, 1, 3, 1),\phi(x_1, 3, 3, 1, 3, 1) \}\\
&=& \max\left\{\frac{-x_1 + 19}{x_1 + 17},
\frac{-x_1^2 + 10x_1 + 7}{x_1^2 + 6x_1 + 9},\frac{-x_1^2 + 12x_1 + 13}{(x_1 + 3)\sqrt{x_1^2 + 18x_1 +17}}\right\}.
\end{eqnarray*} Note that the right hand side of the equation is less than 1 for {$x_1\in (1,3].$} This yields that $\phi(x)<1.$

For the lower bound estimate of $\phi,$ set $\phi^-(x)=-(\phi(x)\wedge 0),$
i.e. the negative part of $\phi(x).$ It suffices to prove that $\phi^-(x)<1.$
We may assume that $\phi(x)<0,$ otherwise it is trivial. In this case,
\begin{eqnarray*}
0<\phi^-(x) &=&
\frac{ (x ^2_1-1)x_4-x_2x_3 - x_5x_6 - x_1x_2x_5 - x_1x_3x_6}{\sqrt {2x_1x_2x_6 + x^2_1 + x^2_2 + x^2_6 - 1}\sqrt {2x_1x_3x_5 + x^2_1 + x^2_3 + x^2_5 - 1}}\\
&<& \frac{3x_1^2-2x_1-5}{x_1^2+2x_1+1},
\end{eqnarray*} where we use that the right hand side of the first equation is monotonely increasing in $x_4$ and decreasing in $x_2,x_3,x_5$ and $x_6.$ The result follows from
$$\frac{3x_1^2-2x_1-5}{x_1^2+2x_1+1}\leq 1,\quad x_1\in (1,3].$$


By Proposition~\ref{REAL}, the second assertion follows from the first one.
\end{proof}

Note that $\phi_{ij}$ and $\alpha_{ij},$ defined in \eqref{defi:dihe}, are continuous functions on $\R^6_{\geq 0}.$ In particular, for fixed $x_2,\cdots,x_6,$ $$\phi\to 1,\quad \mathrm{as}\ \ x_1\to 1.$$ This yields that $\alpha_1\to 0.$ In the following, we give a quantitative estimate.
\begin{prop}\label{prop:c0}
For any $C>1$ and $\epsilon>0,$ there exists $\delta(C,\epsilon)>0$ such that
for $$\max_{2\leq i\leq 6}\{x_i\}\leq C\quad \mathrm{and}\quad \ x_1\leq 1+\delta,$$ we have $$\phi(x)\geq \cos\epsilon.$$ In particular, $\alpha_1\leq \epsilon.$  For the case that $C=3$ and {$\epsilon < \pi,$} we can choose $$\delta(3,\epsilon)=\frac{1}{5\pi^2}\epsilon^2.$$ \end{prop}
\begin{proof}
By \eqref{PHI}, noting that $x_i\in [1,C],\forall\ 2\leq i\leq 6$ and $x_1\in [1,1+\delta],$
\begin{eqnarray*}
\phi(x) &\geq&
\frac{x_2x_3 + x_5x_6 + x_2x_5 + x_3x_6 - (2\delta+\delta^2)C}{\sqrt {2x_1x_2x_6 + x^2_1 + x^2_2 + x^2_6 - 1}\sqrt {2x_1x_3x_5 + x^2_1 + x^2_3 + x^2_5 - 1}}.
\end{eqnarray*} Choose $\delta<\frac{1}{C}$ such that the numerator is positive. Hence
\begin{eqnarray}
\phi(x) &\geq&
\frac{(x_2+x_6)(x_3+x_5) - (2\delta+\delta^2)C}{\sqrt {(x_2+x_6)^2 + 2\delta C^2 + 2\delta+\delta^2}\sqrt {(x_3+x_5)^2 + 2\delta C^2 + 2\delta+\delta^2}}\nonumber\\
&=&\frac{1 - \frac{(2\delta+\delta^2)C}{(x_2+x_6)(x_3+x_5)}}{\sqrt {1 + (2\delta C^2 + 2\delta+\delta^2)/(x_2+x_6)^2}\sqrt {1 + (2\delta C^2 + 2\delta+\delta^2)/(x_3+x_5)^2}}\nonumber\\
&\geq &\frac{1 - \frac14(2\delta+\delta^2)C}{1 + \frac14(2\delta C^2 + 2\delta+\delta^2)}\to 1,\quad\quad \mathrm{as}\ \delta\to 0.\label{eq:del1}
\end{eqnarray}
Hence for any $\epsilon>0,$ there exists $\delta(C,\epsilon)>0$ such that
$$\phi(x)\geq \cos\epsilon.$$ This proves the first assertion.

For the second assertion, for $C=3$ and {$\epsilon < \pi,$} we choose
$\delta=\frac{1}{5\pi^2}\epsilon^2<\frac{1}{5}.$
By \eqref{eq:del1},
\begin{eqnarray*}
\phi(x) &\geq&\frac{1 - \frac34(2\delta+\delta^2)}{1 + \frac14(20\delta+\delta^2)}\\
&\geq&\frac{1 - \frac94\delta}{1 + \frac{21}{4}\delta}\geq 1-\frac{2}{\pi^2}\epsilon^2.
\end{eqnarray*} Note that for any {$0 < \epsilon < \pi,$}
\begin{equation*}1-\cos\epsilon=2\sin^2(\frac{\epsilon}{2})\geq\frac{2}{\pi^2}\epsilon^2.
\end{equation*}
Hence $\phi(x)\geq \cos\epsilon.$ This proves the second assertion.

\end{proof}

\section{Extended Ricci flows}\label{sec:ecf}


Let $\mathscr{T}= \{T_1,\cdots, T_t
\},$ $t\in \N$, be a finite collection of combinatorial tetrahedra. Let $(M,\T)$ be a closed pseudo 3-manifold, which is the quotient space of $\mathscr{T}=T_1 \sqcup\cdots \sqcup T_t,$ a simplicial complex of the disjoint union of tetrahedra, via affine isomorphisms pairing faces of
tetrahedra. We denote by $E(\cdot)$ ($T(\cdot),$ resp.) the set of edges (the set of tetrahedra, resp.) in $(\cdot).$ For any $e\in E(\cdot)$ and $\sigma\in T(\cdot),$ we say that $e$ is \emph{incident} to $\sigma,$ denoted by $e\sim \sigma$, if the former is contained in the latter.

We denote by
$$P_E: E(\mathscr{T})\to E(\T),\quad P_T: T(\mathscr{T})\to T(\T)$$ the projection maps (or the quotient maps) on the set of edges and tetrahedra respectively.
Note that for any $e\in E(\T),$ the valence of $e,$ $d_e$, is given by the number of edges in $P_E^{-1}(e).$ For $l\in \R^{E(\T)}_{>0},$ it induces the length vector on $E(\mathscr{T})$ via
$$\hat{l}:=l\circ P_E: E(\mathscr{T})\to \R_{>0}.$$ We endow each $T_j\in \mathscr{T}$ with the structure of generalized hyper-ideal tetrahedron by the length vector $\hat{l}.$ For any $\hat{e}\in E(\mathscr{T}),$ it is contained in a unique tetrahedron $T_j.$ We denote by $\alpha(\hat{e})$ the extended dihedral angle of $\hat{e}$ in $T_j$ with respect to the length vector $\hat{l},$ which is given by Definition~\ref{defi:dihe}.

\begin{defi}\label{defi:pre}For any $l\in \R^{E(\T)}_{>0},$ the \emph{generalized curvature} of an edge $e\in E(\T)$ is defined as
$$\wt{K}_e(l)=2\pi-\sum_{\hat{e}\in P_E^{-1}(e)} \alpha(\hat{e}).$$
\end{defi}
\begin{rem}\begin{enumerate}\item This extends the definition of the Ricci curvature of an edge for hyper-ideal metrics, see Definition~\ref{defi:realcurv}, i.e., for any $l\in \mathcal{L}(M,\T),$
$$\wt{K}_e(l)=K_e(l),\quad \forall e\in E(\T).$$ \item For any $l\in \R^{E(\T)}_{\geq 0},$ the dihedral angles are well defined; see Definition~\ref{defi:dihe}. Hence the generalized Ricci curvature can be extended to $\R^{E(\T)}_{\geq 0},$ still denoted by $\wt{K}.$
\end{enumerate}\end{rem}

We prove the regularity for the extended curvature function $\wt{K}(l).$
\begin{prop}\label{prop:lipc} The extended Ricci curvature $\wt{K}(l)$ is locally Lipschitz in $\R^{E(\T)}_{>0}.$
\end{prop}
\begin{proof} By the definition, it suffices to prove that for a tetrahedron $\hat{\sigma}\in T(\mathscr{T})$ with the generalized hyper-ideal metric given by $\hat{l}=l\circ P_E,$ for any $\hat{e}\sim \hat{\sigma},$
$\alpha(\hat{e}),$ as a function of $l,$ is locally Lipschitz on $\R^{E(\T)}_{>0}.$

Note that for a single tetrahedron $\tau=\{1,2,3,4\}$ with a generalized hyper-ideal metric $(l_{12},\cdots,l_{34})\in \R^6_{>0}$ and any $\{i,j\}\subset\{1,2,3,4\},$
the function $\phi_{ij},$ defined in \eqref{angle}, is locally Lipschitz for $(l_{12},\cdots,l_{34})\in \R^6_{>0}.$ By Definition~\ref{defi:dihe}, $\alpha_{ij}$ is also locally Lipschitz for $(l_{12},\cdots,l_{34})\in \R^6_{>0}.$

By the above observation, for the tetrahedron $\hat{\sigma},$ $\alpha(\hat{e})$ is a locally Lipschitz function in the variables of its edge lengths, which is given by $\hat{l}.$ Hence the extended Ricci curvature $\wt{K}(l)$ is locally Lipschitz for $l\in\R^{E(\T)}_{>0}.$
\end{proof}

Let $E(\T)=\{e_1,e_2,\cdots, e_m\},$ where $m$ is the number of edges. To simplify the notation, we write $E=E(\T)=\{1,2,\cdots, m\},$ that is, each edge $e_i$ is replaced by the index $i.$ In this way, the edge length vector is written as $$l=(l_1,l_2,\cdots l_m):=(l(e_1),l(e_2),\cdots, l(e_m)).$$ Given a generalized hyper-ideal metric $l\in\R^E_{>0},$ the generalized Ricci curvature at edges can be written as a vector \[\wt{K}(l)=(\wt{K}_1(l), \wt{K}_2(l),\cdots, \wt{K}_m(l)).\]
This yields the curvature map $$\wt{K}:\R^E_{>0}\to \R^E,\quad l\mapsto \wt{K}(l).$$ If $l$ is a hyper-ideal metric, i.e., $l\in \mathcal{L}(M,\T),$ we write $K$ as above instead of $\wt{K}.$

For simplicity, we write $T=T(\T).$ Note that the map $P_T$ is a bijection. For any $\sigma \in T,$ we write $$\hat{\sigma}=P_T^{-1}(\sigma).$$


For a closed pseudo 3-manifold $(M, \mathcal{T}),$ following Luo and Yang \cite{[LY]},  we define the co-volume functional as
\[cov:\R^{E}\to\R,\quad \quad cov(l) = \sum_{\sigma \in T} cov_{\hat{\sigma}}(\hat{l})\]
{where $cov_{\hat{\sigma}}$ denotes the co-volume functional of a tetrahedron $\hat{\sigma}$ defined in \eqref{cov}. By Proposition~\ref{covx} and Proposition~\ref{prop:strictconvex}, $cov_{\hat{\sigma}}$ is a $C^1$-smooth convex function on $\R^6,$ which is locally strictly convex in $\mathcal{L}.$ This yields the following proposition.
\begin{prop}\label{prop:pp1} The functional $cov$ is a $C^1$-smooth convex function on $\R^E,$ which is smooth and locally strictly convex on $\mathcal{L}(M, \mathcal{T}).$
\end{prop}

For our purposes, we introduce the following functional.
\begin{defi} The functional $H:\R^E\to\R$ is defined as
 $$H(l) = cov(l) - 2\pi\sum_{i\in E} l_i,\quad l\in \R^E.$$
\end{defi}

By Proposition~\ref{prop:pp1}, we have the following.
\begin{prop}\label{prop:stc} The functional $H$ is a $C^1$-smooth convex function on $\R^E,$ which is smooth and locally strictly convex on $\mathcal{L}(M, \mathcal{T}).$
\end{prop}
For a tetrahedron $\hat{\sigma}\in T(\mathscr{T})$ with the generalized hyper-ideal metric given by $\hat{l},$ by the definitions of $cov,$ \eqref{cov} and \eqref{eq:mu1}, for any $\hat{e}\sim \hat{\sigma},$
$$\frac{\partial cov_{\hat{\sigma}}(\hat{l})}{\partial (\hat{l}(\hat{e}))}=\alpha(\hat{e}).$$
Hence for any $j\in E,$
\begin{eqnarray}
\frac{\partial H}{\partial l_j}
&=& \frac{\partial cov(l)}{\partial l_j} - 2\pi
=\frac{\partial}{\partial l_j}\left(\sum_{\sigma \in T} cov_{\hat{\sigma}}(\hat{l})\right)-2\pi\label{eq:deriv}\\
&=&\sum_{\substack{\hat{e}\in P_E^{-1}(j),\nonumber\\
 \hat{e}\sim \hat{\sigma}}}\frac{\partial cov_{\hat{\sigma}}(\hat{l})}{\partial (\hat{l}(\hat{e}))}\frac{\partial (\hat{l}(\hat{e}))}{\partial l_j}-2\pi\nonumber\\
&=& \sum_{\hat{e}\in P_E^{-1}(j)}\alpha(\hat e) - 2\pi\nonumber\\
&=& -\wt{K}_j,\nonumber
\end{eqnarray} where we have used $\hat{l}(\hat{e})=l_j$ for $\hat{e}\in P_E^{-1}(j).$
Since the generalized Ricci curvature $\wt{K}$ extends the Ricci curvature $K,$
for any $l\in \mathcal{L}(M,\T),$
\begin{equation}\label{eq:deriv1}
\frac{\partial H}{\partial l_j}=-K_j,\quad \forall j\in E.
\end{equation}

We introduce a combinatorial Ricci flow in $\mathcal{L}(M, \mathcal{T})$, an analog of Luo's Ricci flow \eqref{eq:luoflow},
\beq
\begin{cases}\label{kl}
\frac{d l_i(t)}{d t} = K_i(l(t)) l_i(t),\quad \forall i\in E, t\geq 0,\\
l(0) = l_0,
\end{cases}
\eeq
where $l_0\in\mathcal{L}(M, \mathcal{T})$ and $l(t)\in \mathcal{L}(M, \mathcal{T}),\forall t>0.$ Since $\mathcal{L}(M, \mathcal{T})$ is an open subset in $\R^E_{>0}$ and $K(l)$ is a smooth function on $\mathcal{L}(M, \mathcal{T}),$ Picard's theorem in ordinary differential equations yields the following.
\tm
For a closed pseudo 3-manifold $(M, \mathcal{T}),$ for any initial data $l_0\in\mathcal{L}(M, \mathcal{T}),$ the solution $\{l(t)|t\in [0, T )\}\subset \mathcal{L}(M, \mathcal{T})$ to the combinatorial Ricci flow \eqref{kl} exists and is unique on the maximal existence interval $[0, T )$ with $0 < T\leq \infty.$

\tmd

Now, we introduce the extended Ricci flow on $\R^E_{>0},$ see \eqref{eq:newflow},

\beq
\begin{cases}\label{exkl}
\frac{d l_i(t)}{d t} = \wt K_i(l(t)) l_i(t),\quad \forall i\in E, t\geq 0,\\
l(0) = l_0,
\end{cases}
\eeq
where $l_0\in\R^E_{>0}$ and $l(t)\in \R^E_{>0},\forall t>0.$

We prove the long-time existence and uniqueness of the extended Ricci flow.
\begin{proof}[Proof of Theorem~\ref{thm:gp1}]
Note that $\R^E_{>0}$ is an open subset in $\R^E$ and $\wt{K}(l)$ is a locally Lipschitz function on $\R^E_{>0};$ see Proposition~\ref{prop:lipc}. Picard's theorem in ordinary differential equations yields the local existence and uniqueness of the extended Ricci flow \eqref{exkl}, i.e. for any $l_0\in\R^E_{>0},$ the solution $\{l(t)|t\in [0, T )\}\subset \R^E_{>0}$ to the extended Ricci flow \eqref{exkl} exists and is unique on the maximal existence interval $[0, T )$ with $0 < T\leq \infty.$

It suffices to prove that $T=\infty.$
For any $l\in \R^E_{>0},$ by Definition~\ref{defi:pre},
$$2\pi-\pi d_i\leq \wt{K}_i(l)\leq 2\pi,\quad \forall i\in E,$$ where $d_i$ is the valence of the edge $i.$ Hence there exists a constant $C,$ depending on the triangulation $\T,$ such that for any $l\in \R^E_{>0},$
$$|\wt{K}_i(l)|\leq C,\quad \forall i\in E.$$ Hence, for any $t\in [0, T ),$
$$\left|\frac{d l_i(t)}{d t}\right|= |\wt K_i(l(t)) l_i(t)|\leq Cl_i(t),\quad \forall i\in E.$$
This yields that $$l_i(0)e^{-Ct}\leq l_i(t)\leq l_i(0)e^{Ct}, \quad \forall i\in E, t\geq 0.$$
By using a contradiction argument and the local existence result for any initial data, one can show that $T=\infty.$ This proves the result.

\end{proof}

By \eqref{eq:deriv}, the extended Ricci flow can be rewritten as
$$\frac{d l_i}{d t} = \wt K_i(l) l_i=-\frac{\partial H}{\partial l_i} l_i,\quad \forall i\in E, t\geq 0.$$ This implies that the extended Ricci flow is a variant of the negative gradient flow associated with the functional $H.$

\begin{prop}\label{H}
The functional $H$ is non-increasing along the extended Ricci flow (\ref{exkl}), i.e. for any solution $l(t)$ to the extended Ricci flow \eqref{exkl},
\[\frac{d H(l(t))}{d t} \le 0.\]
\end{prop}
\begin{proof}
By direct calculation and \eqref{eq:deriv},
$$\frac{d H}{d t} = -\sum_i \wt{K}_i^2 l_i \le 0.$$
\end{proof}

In the following, we prove that if the extended Ricci flow converges, then the limit is a generalized hyper-ideal metric with zero Ricci curvature.
\tm\label{thm:zero1}
Let $\{l(t)| t\in[0,\infty)\}$ be the solution to the extended Ricci flow \eqref{exkl}. Suppose that there exists $\bar l\in \R^E_{> 0}$ such that $$l(t)\to \bar{l},\quad t\to\infty.$$ Then
$\wt{K}(\bar l) = 0$.
\tmd
\begin{proof}
By Proposition \ref{H}, $H(l(t))$ is non-increasing. Note that $$l(t) \to \bar l, t \to \infty,\quad \bar l\in \R^E_{> 0}.$$ Since $H$ is continuous on $\R^E_{>0},$ the set
$\{H(l(t)) : t \ge 0\}\subset \R$ is bounded.
Hence the following limit exists and is finite,
\[\lim_{t \to \infty} H(l(t)) =A.\]
Consider the sequence $\{H(l(n))\}^\infty_{n=1}$. By the mean value theorem, for any
$n \ge 1$ there exists $t_n \in (n, n + 1)$ such that
\beq\label{eq:qqq1}
H(l(n + 1)) - H(l(n)) =\frac{d}{dt}\big |_{t = t_n} H(l(t)) = - \sum_i \wt{K}_i^2(l(t_n))l_i(t_n).
\eeq

Note that $\lim_{n \to \infty} H(l(n + 1)) - H(l(n)) = 0$. By passing to the limit, $n\to\infty,$ in \eqref{eq:qqq1}, noting that $\wt{K}(l)$ is continuous in $\R^E_{>0},$ we have
$$ - \sum_i \wt{K}_i^2(\bar{l})\bar{l}_i=0.$$ This yields that $\wt{K}(\bar l)=0.$

\end{proof}


The uniqueness of the hyper-ideal metric with zero Ricci curvature was proved by Luo and Yang \cite[Theorem~1.2]{[LY]}. {Similar rigidity results can be seen in \cite{[Cho],[Luo2],[W]}.}
\tm\label{thm:uniquezero} Let $(M, \mathcal{T})$ be a closed pseudo 3-manifold. Suppose that there exists a hyper-ideal metric $\hat{l} \in \mathcal{L}(M, \mathcal{T})$ such that $K(l) = 0$, then $l$ is the unique metric with zero Ricci curvature in the class $\R^E_{\geq 0},$ i.e. if $\wt{K}(\bar l) = 0$ for some $\bar l\in \R^E_{\geq 0},$ then $\bar l=l.$ Moreover, $$\lim_{l\in \R^E, l\to\infty} H(l)=+\infty.$$
\tmd

\begin{proof} For readers' convenience, we include the proof here.
We claim that for any $\bar l\in \R^E_{\geq 0}\setminus \R^{E}_{>0},$
$\wt{K}(\bar l)\neq 0.$ Note that there exists some $i\in E$ such that $\bar l_i=0.$ For any $\hat{e}\in P_E^{-1}(i),$ and $\hat{e}\sim \hat{\sigma}\in T(\mathscr{T}),$ by \eqref{angle} and  Definition~\ref{defi:dihe}, we get $\alpha(\hat{e})=0.$ This yields
$$\wt{K}_i(\bar l)=2\pi-\sum_{\hat{e}\in P_E^{-1}(e)}\alpha(\hat{e})=2\pi\neq 0.$$ This proves the claim.

Hence it suffices to prove the result for $\bar l\in \R^{E}_{>0}.$
We consider the functional $H:\R^E\to \R.$ By Proposition~\ref{prop:stc}, it is $C^1$-smooth and convex on $\R^E,$ and is smooth and strictly convex on $\mathcal{L}(M, \mathcal{T}).$ By \eqref{eq:deriv}, the metrics in $\R^E_{>0}$ with zero Ricci curvature correspond to the critical points of the functional $H.$ Note that any critical point of a convex function on a convex domain is a minimizer. So that the zero-curvature metric $\hat{l}$ is a minimizer of $H.$ Moreover, $l\in \mathcal{L}(M, \mathcal{T})$ and $H$ is strictly convex on $\mathcal{L}(M, \mathcal{T}).$ This is a unique minimizer, also a unique critical point, of $H$ on $\R^E.$  By the strict convexity on $\mathcal{L}(M, \mathcal{T}),$ we have $$\lim_{l\in \R^E, l\to\infty} H(l)=+\infty.$$ This proves the result.
\end{proof}

\section{Proof of Main Theorems}
In this section, we prove the main results of the paper. 
For that purpose, we first obtain the uniform bounds for the solution to the extended Ricci flow \eqref{exkl} under some proper conditions on combinatorial information of the triangulation and the initial data.

In the following, we prove the upper bound estimate for the solution to the extended Ricci flow.

\begin{theo}\label{LONG}
Let $(M, \T)$ be a closed pseudo 3-manifold satisfying that $d_i \ge 10$ for all $i\in E.$ Let $\{l(t)|t\in[0,\infty)\}\subset \R^E_{>0}$ be the solution to the extended Ricci flow \eqref{exkl} with initial data $l_0\in (0,\arccosh 3)^E.$ Then for any $t\geq 0,$ $$l(t)\in (0,\arccosh 3)^E.$$ In particular, $l(t)\in \mathcal{L}(M,\T),$ for any $t\geq 0.$
\end{theo}
\begin{proof} It suffices to prove that for any $t\geq 0,$
$$l(t)<\arccosh 3.$$
Suppose that it is not true, then
$$K:=\{t\in [0,\infty): \max_{i\in E} l_i(t)\geq \arccosh 3\}$$ is a non-empty, closed subset of $\R.$ Set $t_0:=\inf K.$ Since $l_0\in (0,\arccosh 3)^E,$ $0\not\in K.$ Hence $0<t_0<\infty.$ There exists some $j\in E$ such that
$$l_j(t_0)=\max_{i\in E} l_i(t_0).$$ Note that $l_j(t_0)=\arccosh 3$ and $$l_j(t)\leq \max_{i\in E} l_i(t)<\arccosh 3,\quad \forall\ t<t_0.$$ This yields that $$l_j'(t_0)\geq 0.$$ For any $\hat{e}\in P_E^{-1}(j),$ and $\hat{e}\sim \hat{\sigma}\in T(\mathscr{T}),$ by Corollary~\ref{lem:longest} and $l_j(t_0)=\max_{i\in E} l_i(t_0)=\arccosh 3,$ $$\alpha(\hat{e})>\frac{\pi}{5}.$$
By $d_i \ge 10$ for all $i\in E,$ we have
$$\wt{K}_j(l(t_0))=2\pi-\sum_{\hat{e}\in P_E^{-1}(e)}\alpha(\hat{e})<2\pi-\frac{\pi}{5}d_j\leq 0.$$ This yields the following contradiction, $$0\leq l_j'(t_0)=\wt{K}_j (l(t_0))l_j(t_0)<0.$$ This proves the first statement.

The last statement follows from Theorem~\ref{Range}.
\end{proof}

Next, we prove the lower bound estimate for the solution to the extended Ricci flow.
\begin{theo}\label{sbd}
Let $(M, \T)$ be a closed pseudo 3-manifold, and $\{l(t)|t\in[0,\infty)\}\subset \R^E_{>0}$ be the solution to the extended Ricci flow \eqref{exkl} with initial data $l_0.$ Suppose that there exists a constant $C$ such that $$l(t)\in (0,C]^E,\quad \forall t\geq 0.$$ Then there exists a positive constant
$c(l_0,C,\max_{i\in E}d_i),$ depending on $l_0,$ $C$ and $\max_{i\in E}d_i,$ such that 
\begin{equation}\label{eq:spq1}l(t)\in (c,C]^E,\quad \forall t\geq 0.\end{equation}
\end{theo}
\begin{proof} 
Set $\epsilon_0:=\frac{\pi}{\max_{i\in E}d_i}.$
Take $$c:=\min\left\{\frac12 \min_{i\in E}l_{0,i},\arccosh(1+\delta)\right\},$$ where $\delta=\delta(\cosh C,\epsilon_0)$ is the constant given in Proposition~\ref{prop:c0}.  We want to prove that \eqref{eq:spq1} holds for the above constant $c.$ Suppose that it is not true, then
$$Q:=\{t\in [0,\infty): \min_{i\in E} l_i(t)\leq c\}$$ is a non-empty, closed subset of $\R.$ Set $t_0:=\inf Q.$ Since \eqref{eq:spq1} holds for $t=0,$ $0\not\in Q.$ Hence $0<t_0<\infty.$ There exists some $j\in E$ such that
$$l_j(t_0)=\min_{i\in E} l_i(t_0).$$ Note that $l_j(t_0)=c$ and $$l_j(t)\geq \min_{i\in E} l_i(t)>c,\quad \forall\ t<t_0.$$ This yields that $$l_j'(t_0)\leq 0.$$ Consider any $\hat{e}\in P_E^{-1}(j),$ and $\hat{e}\sim \hat{\sigma}\in T(\mathscr{T}),$ endowed with the metric given by $\hat{l}(t_0)=l(t_0)\circ P_E.$ Note that $l(t_0)\in (0,C]^E$ and $$\cosh(l_j(t_0))=\cosh  c \le 1+\delta.$$ By Proposition~\ref{prop:c0}, we have under the metric  $l(t_0),$ $$\alpha(\hat{e})<\epsilon_0.$$
Hence
$$\wt{K}_j(l(t_0))=2\pi-\sum_{\hat{e}\in P_E^{-1}(e)}\alpha(\hat{e})>2\pi-d_j\epsilon_0\geq \pi> 0.$$ This yields the following contradiction, $$0\geq l_j'(t_0)=\wt{K}_j (l(t_0))l_j(t_0)>0.$$ This proves the result.

\end{proof}

By the same argument as above, we can prove a quantitative lower bound estimate for the solution to the extended Ricci flow.
We denote $\mathbbm{1}_E$ by the constant function $1$ on $E,$ i.e.
$$\mathbbm{1}_E=(1,1,\cdots,1).$$
\begin{theo}\label{bd}
Let $(M, \T)$ be a closed pseudo 3-manifold satisfying that $d_i \ge 10$ for all $i\in E.$ Let $\{l(t)|t\in[0,\infty)\}\subset \R^E_{>0}$ be the solution to the extended Ricci flow \eqref{exkl} with initial data $l_0\in (0,\arccosh 3)^E.$ Then there exists a positive constant
$c(l_0,\max_{i\in E}d_i)$ such that\begin{equation}\label{eq:pq1}l(t)\in (c,\arccosh 3)^E, \quad \forall t\geq 0.\end{equation}
In particular, if $l_0=\frac{\arccosh 3}{2} \mathbbm{1}_E,$ then the above constant $c$ can be chosen as $$c(l_0,\max_{i\in E}d_i)=\frac{1}{3 (\max_{i\in E}d_i)}.$$ \end{theo}
\begin{proof} By Theorem~\ref{LONG}, $$l(t)\in (0,\arccosh 3)^E,\quad \forall t\geq 0.$$
Applying Theorem~\ref{sbd} for $C=\arccosh 3,$ we get the result \eqref{eq:pq1}.

Let $l_0=\frac{\arccosh 3}{2} \mathbbm{1}_E.$ Set $\epsilon_0:=\frac{\pi}{\max_{i\in E}d_i}.$
By the same argument as in the proof of Theorem~\ref{sbd}, we prove that $$l(t)\in (c_1,\arccosh 3)^E, \quad \forall t\geq 0,$$ where $c_1=\min\left\{\frac{\arccosh 3}{4},\arccosh(1+\delta)\right\}.$ Here $\delta=\delta(3,\epsilon_0)=\frac{1}{5(\max_{i\in E}d_i)^2}\leq \frac{1}{500}.$
Note that $$c_1=\arccosh(1+\delta)\geq \sqrt{\delta}\geq \frac{1}{3\max_{i\in E}d_i}.$$
Hence we can choose $c=\frac{1}{3\max_{i\in E}d_i}$ as a new lower bound in \eqref{eq:pq1}.
\end{proof}

We recall a well-known lemma in the theory of ordinary differential equations.
\begin{lem}[\cite{[P1]}]\label{lem:conv} Let $V$ be an open set in $\R^n$ and $f\in C^1(V,\R^n).$ Consider an autonomous ordinary differential system
\begin{equation}\label{eq:5-1-2}\frac{d}{dt}{{x}(t)}={f}({x}(t)),~~~{x}(t)\in V.\end{equation}
Assuming ${x}^*\in V$ is a critical point of $f$, i.e. ${f}({x}^*)=0$. If all the eigenvalues of the Jacobian matrix $\frac{\partial {f}}{\partial{x}}({x}^*)$ have negative real part, then ${x}^*$ is an asymptotically stable point. More specifically, there exists a neighbourhood $\wt{V}\subset V$ of ${x}^*$, such that for any initial ${x}(0)\in \wt{V}$, the solution ${x}(t)$ to the equation \eqref{eq:5-1-2} exists for all time $t\in[0,\infty)$ and converges exponentially fast to ${x}^*$.
\end{lem}

Next, we prove the characterization of the convergence of the extended Ricci flow. 
\begin{proof}[Proof of Theorem~\ref{thm:expc}] By Theorem~\ref{thm:zero1}, the limit metric has zero Ricci curvature if the extended Ricci flow \eqref{eq:newflow} converges. We prove the other direction. Suppose that there exists a zero-curvature hyper-ideal metric $\hat{l},$ then for any initial data $l_0\in \R^E_{>0}$ the extended Ricci flow $l(t)$ converges to $\hat{l}$ exponentially fast.

By Theorem~\ref{thm:uniquezero}, the functional $H:\R^E\to \R$ is proper and bounded from below.
Since $H(l(t))$ is non-increasing by the same argument as in Proposition \ref{H},
$\{H(l(t)) : t \ge 0\}\subset \R$ is bounded. Hence $\{l(t) : t \ge 0\}$ is a bounded set in $\R^E_{>0}.$ By Theorem~\ref{sbd}, there exist positive constants $c,C$ such that \begin{equation}\label{eq:dd1}l(t)\in [c,C]^E,\quad t\geq 0.\end{equation}

Note that the following limit exists and is finite,
\[\lim_{t \to \infty} H(l(t)) = {A.}\]
Consider the sequence $\{H(l(n))\}^\infty_{n=1}$. By the mean value theorem, for any
$n \ge 1$ there exists $t_n \in (n, n + 1)$ such that
\beq\label{eq:qaaa1}
H(l(n + 1)) - H(l(n)) =\frac{d}{dt}\big |_{t = t_n} H(l(t)) = - \sum_i \wt{K}_i^2(l(t_n))l_i(t_n).
\eeq
Note that $\lim_{n \to \infty} H(l(n + 1)) - H(l(n)) = 0$ and the estimate \eqref{eq:dd1} holds. Passing to the limit, $n\to\infty,$ in \eqref{eq:qaaa1}, we have
\begin{equation}\label{eq:qmm2}\lim_{n\to \infty}|\wt{K}_i(l(t_n))|=0,\quad \forall i\in E.\end{equation} By \eqref{eq:dd1}, there exist a subsequence of $\{t_n\}_{n=1}^\infty,$ denoted by $\{t_{n_k}\}_{k=1}^\infty,$ and $l_\infty\in [c,C]^E$ such that
\begin{equation}\label{eq:qmm4}l(t_{n_k})\to l_\infty,\quad k\to\infty.\end{equation}
By the continuity of $\wt{K}(l)$ and \eqref{eq:qmm2}, we have
$$\wt{K}(l_\infty)=0.$$ 
By Theorem~\ref{thm:uniquezero}, $\hat{l}$ is the unique metric with zero Ricci curvature in the class $\R^E_{\geq 0}.$ Hence $l_\infty=\hat{l}.$

Next, we prove the exponential convergence of the extended Ricci flow $l(t)$ to $\hat{l}.$
Let $m=|E|.$ We define $f=(f_1,\cdots,f_m): \R^E_{>0}\to \R^m$ as
$$f_i(l)=\wt{K}_i(l) l_i,\quad \forall\ l\in \R^E_{>0}, i\in E.$$
Then the extended Ricci flow \eqref{exkl} can be written as \begin{equation}\label{eq:5-1-1}\frac{d}{dt}l(t)={f}(l(t)).\end{equation}
Note that the set of critical points of ${f}$ consists of a single point $\hat{l}.$ We calculate the Jacobian matrix of the map ${f}$ at $\hat{l}.$ Since $\hat{l}\in \mathcal{L}(M,\T),$ for any $l$ in a small neighbourhood $V$ of $\hat{l},$
$$f_i(l)=\wt{K}_i(l) l_i=K_i(l) l_i,\quad i\in E.$$
Hence
$$\frac{\partial {f}}{\partial l}(\hat{l})=\Sigma \left(\frac{\partial  K}{\partial l}(\hat{l})\right),$$ where $\Sigma=\mathrm{diag}\{\hat{l}_{1},\cdots,\hat{l}_{m}\}.$ Note that by Proposition~\ref{prop:strictconvex} and \eqref{eq:deriv1},
$\frac{\partial  K}{\partial l}(\hat{l})$ is a negative definite matrix.
 By Lemma~\ref{lem:conv},
$\hat{l}$ is an asymptotically stable point of the flow \eqref{eq:5-1-1}, which is equivalent to the extended Ricci flow \eqref{exkl}. That is, there exists a neighbourhood $\wt{V}\subset V$ of $\hat{l}$, such that for any initial $\wt{l}_0\in \wt{V}$, the solution $\wt{l}(t)$ to the equation \eqref{eq:5-1-1} exists for all time $t\in[0,\infty)$ and converges exponentially fast to $\hat{l}$. By \eqref{eq:qmm4}, there exists a sufficiently large time $t_1$ such that
$$l(t_1)\in \wt{V}.$$ Consider the solution $\wt{l}(t)$ of the flow \eqref{eq:5-1-1} with the initial data $l(t_1),$ which converges exponentially fast to $\hat{l}.$ By the uniqueness of the solution, Theorem~\ref{thm:gp1}, $\wt{l}(t)=l(t+t_1)$ for all $t\geq 0.$ Hence the solution $l(t)$ converges exponentially fast to $\hat{l}$. This proves the result.


\end{proof}

Now we are ready to prove Theorem~\ref{thm:main10}.
\begin{proof}[Proof of Theorem~\ref{thm:main10}]
For any initial data $l_0\in (0,\arccosh 3)^E,$ let $l(t)$ be the solution to the extended Ricci flow \eqref{exkl}. By Theorem~\ref{LONG} and Theorem~\ref{bd}, there exist a constant $C_1=c(l_0,\max_{i\in E}d_i)$ such that
\begin{equation}\label{eq:mm1}l(t)\in (C_1,C_2)^E,\quad \forall t\geq 0,\end{equation} where $C_2=\arccosh 3.$

We first prove that the existence of the hyper-ideal metric with zero Ricci curvature.
By \eqref{eq:mm1} and the continuity of $H$, the set
$\{H(l(t)) : t \ge 0\}\subset \R$ is bounded.
By the monotonicity of $H(l(t)),$ the following limit exists and is finite,
\[\lim_{t \to \infty} H(l(t)) = A.\]
Consider the sequence $\{H(l(n))\}^\infty_{n=1}$. By the mean value theorem, for any
$n \ge 1$ there exists $t_n \in (n, n + 1)$ such that
\beq\label{eq:aaa1}
H(l(n + 1)) - H(l(n)) =\frac{d}{dt}\big |_{t = t_n} H(l(t)) = - \sum_i \wt{K}_i^2(l(t_n))l_i(t_n).
\eeq
Note that $\lim_{n \to \infty} H(l(n + 1)) - H(l(n)) = 0$ and we have the estimate \eqref{eq:mm1}. Passing to the limit, $n\to\infty,$ in \eqref{eq:aaa1}, we have
\begin{equation}\label{eq:mm2}\lim_{n\to \infty}|\wt{K}_i(l(t_n))|=0,\quad \forall i\in E.\end{equation} By \eqref{eq:mm1}, there exist a subsequence of $\{t_n\}_{n=1}^\infty,$ denoted by $\{t_{n_k}\}_{k=1}^\infty,$ and $l_\infty\in [C_1,C_2]^E$ such that
\begin{equation}\label{eq:mm4}l(t_{n_k})\to l_\infty,\quad k\to\infty.\end{equation}
By the continuity of $\wt{K}(l)$ and \eqref{eq:mm2}, we have
$$\wt{K}(l_\infty)=0.$$ Since $l_\infty\in [C_1,C_2]^E,$ $l_\infty\in \mathcal{L}(M,\T)$ by Theorem~\ref{Range}. 

By Theorem~\ref{thm:expc}, for any initial data the extended Ricci flow \eqref{exkl} converges to $l_\infty.$

To obtain a refined estimate for the metric $l_\infty,$ we choose a specific initial data $\bar l_0=\frac{\arccosh 3}{2} \mathbbm{1}_E.$ Let $\bar l(t)$ be the solution to the extended Ricci flow \eqref{exkl}. By Theorem~\ref{LONG} and Theorem~\ref{bd}, we have the estimate
 $$\bar l(t)\in ((3 \max_{i\in E}d_i)^{-1}, C_2)^E,\quad \forall\ t\geq 0.$$ Passing to the limit, we get
 $$l_\infty\in [(3 \max_{i\in E}d_i)^{-1}, C_2]^E.$$

This proves the theorem.



\end{proof}


Now we are ready to prove Theorem~\ref{thm:main0}.
{\begin{proof}[Proof of Theorem~\ref{thm:main0}]
Since $(M,\T)$ is a pseudo 3-manifold, by Theorem~\ref{thm:main10}, there exists a zero-curvature hyper-ideal metric $l\in \mathcal{L}(M, \mathcal{T}).$
Thus, we obtain a hyperbolic metric with totally geodesic boundary on $M$ given by the metric space $S(M,\T,l)$ constructed in Definition~\ref{def:t1}. The uniqueness of the metric follows from Mostow rigidity theorem for hyperbolic manifolds with totally geodesic boundary by \cite{[Fri04]}.
 This proves the result.
\end{proof}}

\textbf{Acknowledgements.} We thank Jiming Ma, Yi Liu for many discussions on related problems in this paper. We shared our results to Feng Luo and Tian Yang in earlier times, and we thank them for helpful comments. 

F. K. is supported by NSFC, no. 11901009, H. G. is supported by NSFC, no. 11871094. B. H. is supported by NSFC, no.11831004 and no. 11926313.

\bibliography{hyper1}
\bibliographystyle{alpha}

\end{document}